\documentclass[a4paper]{article}

\usepackage{amsmath,amssymb,color}

\newcommand{\bu}{\boldsymbol u}
\newcommand{\Om}{\Omega}

\newcommand{\bv}{\boldsymbol v}

\newcommand{\bw}{\boldsymbol w}
\newcommand{\bbu}{\boldsymbol b}

\newcommand{\btau}{\boldsymbol \tau}

\newcommand{\beps}{\boldsymbol \varepsilon}

\newcommand{\be}{\boldsymbol e}
\newcommand{\bE}{\boldsymbol E}
\newcommand{\bvar}{\boldsymbol \varphi}

\newcommand{\bs}{\boldsymbol s}

\newcommand{\bff}{\boldsymbol f}

\usepackage{hyperref}

\usepackage{lipsum}
\usepackage{amsfonts}
\usepackage{graphicx}
\usepackage{epstopdf}
\usepackage{algorithmic}
\ifpdf
  \DeclareGraphicsExtensions{.eps,.pdf,.png,.jpg}
\else
  \DeclareGraphicsExtensions{.eps}
\fi

\newtheorem{Theorem}{Theorem}[section]
\newtheorem{lema}[Theorem]{Lemma}

\newtheorem{remark}[Theorem]{Remark}

\newtheorem{Proof}{{\em Proof:}}

\newenvironment{proof}{\begin{Proof}\rm}{\hfill $\Box$ \end{Proof}}

\title{Error analysis of fully discrete mixed finite element data assimilation schemes for the Navier-Stokes equations 
}

\author{ Bosco
Garc\'{i}a-Archilla\thanks{Departamento de Matem\'atica Aplicada
II, Universidad de Sevilla, Sevilla, Spain. Research is supported by
Spanish MINECO under grant MTM2015-65608-P (bosco@esi.us.es).}
  \and Julia Novo\thanks{Departamento de
Matem\'aticas, Universidad Aut\'onoma de Madrid, Spain.  Research is supported
by Spanish MINECO
under grant MTM2016-78995-P (AEI/FEDER, UE) and VA024P17 (Junta de Castilla y Leon, ES) cofinanced by FEDER funds (julia.novo@uam.es).}
}

\usepackage{amsopn}

\ifpdf
\hypersetup{
  pdftitle={Error analysis of fully discrete mixed finite element data assimilation schemes for the Navier-Stokes equations },
  pdfauthor={B. Garc\'{\i}a-Archilla and J. Novo}
}
\fi

\begin{document}

\maketitle


%
\begin{abstract}
In this paper we consider fully discrete approximations with inf-sup stable mixed finite element methods in space to approximate the Navier-Stokes equations. A continuous downscaling data assimilation algorithm is analyzed in which measurements on a coarse scale are given
represented by different types of interpolation operators. For the time discretization an implicit Euler scheme, an implicit and a semi-implicit second order backward
differentiation formula
are considered. Uniform in time error estimates are obtained for all the methods for the error between the fully discrete approximation and the reference solution corresponding to the measurements. For the spatial discretization we consider both the Galerkin method and the Galerkin method with grad-div stabilization. For the last scheme error bounds in which the constants do not depend on inverse powers of the viscosity are obtained.
\end{abstract}

%

76B75. 


\section{Introduction}
Data assimilation refers to a class of techniques that combine experimental data and simulation in order to obtain better predictions in a physical system.
There is a vast literature on data assimilation methods
(see e.g., \cite{Asch_et_al_2016}, \cite{Daley_1993}, \cite{Kalnay_2003}, \cite{Law_Stuart_Zygalakis}, \cite{Reich_Cotter_2015}, and the references
therein). One of these techniques is  nudging in which a penalty term is added with the aim of driving the approximate solution towards coarse mesh observations of the data. In a recent work~\cite{Az_Ol_Ti}, a new approach, known as continuous data assimilation,  is introduced for a large class of dissipative partial differential equations, including Rayleigh-B\'enard convection \cite{FJTi}, the planetary geostrophic ocean dynamics model \cite{FLTi}, etc.
(see also references therein). Continuous data assimilation has also been used in numerical studies, for example, with the Chafee-Infante reaction-diffusion equation,  the Kuramoto-Sivashinsky
equation (in the context of feedback control)~\cite{Lunasin-Titi}, Rayleigh-B\'enard convection equations~\cite{Altaf_et_al}, \cite{FJJTi}, and the Navier-Stokes equations~\cite{Gesho-Olson-Titi}, \cite{HOTi}. However, there is still less numerical analysis of this technique. The present work concerns with the numerical analysis of continuous data assimilation for the Navier-Stokes equations for fully discrete schemes with inf-sup stable mixed finite element methods (MFE) in space.

We consider the Navier-Stokes equations (NSE)
\begin{align}
\label{NS} \partial_t\bu -\nu \Delta \bu + (\bu\cdot\nabla)\bu + \nabla p &= \bff &&\text{in }\ (0,T]\times\Omega,\nonumber\\
\nabla \cdot \bu &=0&&\text{in }\ (0,T]\times\Omega,
\end{align}
in a bounded domain $\Omega \subset {\mathbb R}^d$, $d \in \{2,3\}$. In~\eqref{NS},
$\bu$ is the velocity field, $p$ the kinematic pressure, $\nu>0$ the kinematic viscosity coefficient,
 and $\bff$ represents the accelerations due to external body forces acting
on the fluid. The Navier-Stokes equations \eqref{NS} must be complemented with boundary conditions. For simplicity,
we only consider homogeneous
Dirichlet boundary conditions $\bu = \boldsymbol 0$ on $\partial \Omega$.


As in \cite{Mondaini_Titi} we consider given coarse spatial mesh measurements, corresponding to a solution $\bu$ of \eqref{NS}, observed at a coarse spatial mesh. We assume that the measurements are continuous in time and error-free and we denote by $I_H(\bu)$ the operator used for interpolating these
measurements, where $H$ denotes the resolution of the coarse spatial mesh.
Since no initial condition for $\bu$ is available one cannot  simulate equation \eqref{NS} directly.
To overcome this difficulty it was suggested in~\cite{Az_Ol_Ti} to consider instead a solution~$\bv$ of the following system
\begin{eqnarray}\label{eq:mod_NS}
 \partial_t\bv -\nu \Delta \bv + (\bv\cdot\nabla)\bv + \nabla \tilde p&=&\bff -\beta(I_H(\bv)-I_H(\bu)),\ \text{in }\ (0,T]\times\Omega,\nonumber\\
\nabla \cdot \bv&=&0, \ \text{in }\ (0,T]\times\Omega,
\end{eqnarray}
where $\beta$ is the nudging parameter.

In \cite{Ours}  the continuous in time data assimilation algorithm is analyzed and  two different methods are considered:  the Galerkin method and the Galerkin method  with grad-div stabilization.
In this paper we extend the results in \cite{Ours} to the fully discrete case.
For the time discretization of equation \eqref{eq:mod_NS}, we consider the
fully implicit backward Euler method and the second
order backward differentiation formula (BDF2), both in the fully implicit and
semi-implictit cases.
For the spatial discretization we consider  inf-sup stable mixed finite elements.
As in \cite{Ours}  we consider both  the Galerkin method and the Galerkin method  with grad-div stabilization.
 Although grad-div stabilization
was originally proposed in \cite{FH88} to improve the conservation of mass in
finite element methods, it was observed in the
simulation of turbulent flows in \cite{JK10}, \cite{RL10} that grad-div stabilization has the effect
 of producing stable (non-oscillating) simulations.

  For the three time discretization methods that we consider and the two different spatial discretizations (Galerkin method with or without grad-div stabilization) we prove uniform-in-time error estimates for the approximation of the unknown reference solution, $\bu$, corresponding to the coarse-mesh measurement~$I_H(\bu)$.
 As in \cite{Ours}, for the Galerkin method without stabilization, the spatial error bounds we prove are optimal, in the sense that the rate of convergence is that of the best interpolant. In the case where grad-div stabilization is added, as in~\cite{grad-div1}, \cite{grad-div2}, we get error bounds where the error constants do not depend on inverse powers of~the viscosity parameter~$\nu$. This fact is of importance in many applications where viscosity is orders of magnitude smaller than the velocity
(or with large Reynolds number).

We now comment on the literature on numerical methods for~\eqref{eq:mod_NS}.
In~\cite{Mondaini_Titi}, a semidiscrete postprocessed Galerkin
spectral method
for the two-dimensional Navier-Stokes equations is studied. Under suitable conditions on the nudging parameter~$\beta$, the coarse mesh size~$H$,
and the degrees of freedom in the spectral method, uniform-in-time error estimates are obtained for the error between the numerical approximation to~$\bv$ and~$\bu$. The use of a postprocessing technique introduced in~\cite{Titi1}~\cite{Titi2}, gives a method with a  higher convergence rate than the standard spectral Galerkin  method. A fully-discrete method for the spatial discretization in~\cite{Mondaini_Titi} is analyzed in~\cite{Ibdah_Mondaini_Titi},
where the (implicit and semi-implicit) backward Euler method  is used for time discretization and uniform-in-time
error estimates are obtained with the same convergence rate in space as in~\cite{Mondaini_Titi}.

Other related works are ~\cite{Larios_et_al} and~\cite{Rebholz-Zerfas}.
In~\cite{Rebholz-Zerfas} they only analyze linear problems and, for the proof of the results on the Navier-Stokes equations they
present, they refer to~\cite{Larios_et_al} with some differences that they point out. They also present a wide collection of numerical experiments.
 In~\cite{Larios_et_al},
 the authors consider fully discrete approximations to equation~\eqref{eq:mod_NS}
where the spatial discretization is performed with a MFE Galerkin method with grad-div stabilization. A semi-implicit BDF2 scheme  in time  is analyzed in~\cite{Larios_et_al}, and, as in~\cite{Ibdah_Mondaini_Titi}, \cite{Mondaini_Titi}, \cite{Ours} and the present paper uniform-in-time error bounds are obtained. In the present paper, apart from the semi-implicit BDF2 scheme of~\cite{Larios_et_al} we also analyze the implicit Euler and the implicit BDF2 schemes. Respect to the spatial errors we obtain the same results as in \cite{Ours}. More precisely, comparing with~\cite{Larios_et_al}, we remove the
dependence on inverse powers on $\nu$ on the error constants of the Galerkin method with grad-div stabilization. Also, for the standard Galerkin method, although with constants depending on inverse powers on $\nu$, we get a rate of convergence for the method in space one unit larger than the method in~\cite{Larios_et_al}. These results are sharp in space, as it can be checked in the numerical experiments of \cite{Ours}. This means
that the Galerkin method with grad-div stabilization has a rate of convergence $r$ in the $L^2$ norm of the velocity using polynomials of degree $r$ and that error constants are independent on~$\nu^{-1}$ for the grad-div stabilized method, and dependent in the case of the standard method. The analysis in~\cite{Larios_et_al} is restricted to $I_H\bu$ being an
interpolant for non smooth functions (Cl\'ement, Scott-Zhang, etc),
since explicit use is made
of bounds \eqref{eq:L^2inter} and~\eqref{eq:cotainter}, which are not valid for nodal (Lagrange) interpolation.
 In the present paper, as in \cite{Ours}, we prove error bounds both for the case in which $I_H \bu$ is an interpolant for non smooth functions but also for the case in  which $I_H \bu$ is a standard Lagrange interpolant. To our knowledge reference \cite{Ours} and the present paper are the only references in the literature where such kind of bounds are proved.

It is important to mention that,
compared
 with  \cite{Larios_et_al},~\cite{Rebholz-Zerfas}, \cite{Ibdah_Mondaini_Titi}
 and \cite{Mondaini_Titi}, and as in \cite{Ours}, we do not need to assume an upper bound  on the nudging parameter~$\beta$.  The authors of  \cite{Larios_et_al} had
 observed (see \cite[Remark 3.8]{Larios_et_al}) that the upper bound on~$\beta$ they required in the analysis does not hold in the numerical experiments.
 This fact is corroborated by the numerical experiments  in  \cite{Ours} that show numerical evidence of the advantage of increasing the value of the nudging parameter well above the upper bound in required in previous works in the literature.
The analysis in the present paper, as that in~\cite{Ours}, does not demand any upper bound on the nudging parameter~$\beta$.

 For the error analysis of the fully discrete method with the implicit Euler scheme we do not need to assume any restriction on the size of the time step $\Delta t$ (Theorem~\ref{Th:main_muno0} below). For the case of the implicit BDF2 and semi-implicit BDF2 (Theorems~\ref{th:main_bdf2}
 and~\ref{th:main_bdf2_semimp} below) we only need to assume $\Delta t$ is smaller than a constant (that depends on norms of the theoretical exact solution).

The rest of the paper is as follows. In Section 2 we state some preliminaries and notation. In Section 3 we analyze the fully discrete schemes. First of all, some general results are stated and proved and then the error analysis of the implicit Euler method, the implicit BDF2 method and the
semi-implicit BDF2 schemes is carried out. In Section 4 some numerical experiments are shown.

\section{Preliminaries and Notation}
\label{Se:prelim}
We denote by~$W^{s,p}(D)$ the standard Sobolev space of functions defined on the domain $D\subset\mathbb{R}^d$ with distributional derivatives of order up to $s$ in $L^p(D)$. By~$|\cdot|_{s,p,D}$ we denote the standard seminorm, and, following~\cite{Constantin-Foias}, for~$W^{s,p}(D)$ we
will define the norm~$\|\cdot\|_{s,p,D}$ by
$$
\left\| f\right\|_{s,p,D}^p=\sum_{j=0}^s \left|D\right|^{\frac{p(j-s)}{d}} \left| f\right|_{j,p,D}^p,
$$
where $|D|$ denotes the Lebesgue measure of~$D$. Observe 
that $\left\|f\right\|_{m,p,D}\left|D\right|^{\frac{m}{d}-\frac{1}{p}}$ is scale invariant. If $s$ is not a positive integer, $W^{s,p}(D)$ is defined by interpolation \cite{Adams}.
In the case $s=0$, we have $W^{0,p}(D)=L^p(D)$. As it is customary, $W^{s,p}(D)^d$ will be endowed with the product norm and, since no confusion will arise, it will
also be denoted by $\|\cdot\|_{W^{s,p}(D)}$. When  $p=2$, we will use  $H^s(D)$ to denote the space $W^{s,2}(D)$. By $H_0^1(D)$ we denote the closure in $H^1(D)$ of the set of infinitely differentiable functions with compact support
in $D$.
The inner product of $L^2(\Omega)$ or $L^2(\Omega)^d$ will be denoted by $(\cdot,\cdot)$ and the corresponding norm by $\|\cdot\|_0$.
The norm of the dual space  $H^{-1}(\Omega)$  of $H^1_0(\Omega)$
is denoted by $\|\cdot\|_{-1}$.

We will use the following Sobolev's inequality \cite{Adams}: For~$s>0$, let  $1\le p<d/s$
and~$q$ be such that $\frac{1}{q}
= \frac{1}{p}-\frac{s}{d}$. Then, there exists a positive  scale invariant constant $c_s$ such that
\begin{equation}\label{sob1}
\|v\|_{L^{q'}(\Omega)} \le c_s |\Omega|^{\frac{s}{d}-\frac{1}{p}+\frac{1}{q'}}\| v\|_{W^{s,p}(\Omega)}, \qquad
\frac{1}{q'}
\ge \frac{1}{q}, \quad \forall v \in
W^{s,p}(\Omega).
\end{equation}
If $p>d/s$ the above relation is valid for $q'=\infty$.
A similar inequality holds for vector-valued functions.

We will denote by $\cal H$ and $V$  the Hilbert spaces
$
{\cal H}=\{ \bu \in \big(L^{2}(\Om))^d  \, |\, \mbox{div}(\bu)=0, \,
\bu\cdot n_{|_{\partial \Omega}}=0 \}$,
$V=\{ \bu \in \big(H^{1}_{0}(\Om))^d  \, | \, \mbox{div}(\bu)=0 \}$,
endowed with the inner product of $L^{2}(\Om)^{d}$ and $H^{1}_{0}(\Om)^{d},$
respectively.

The following interpolation inequalities will also be used~(see, e.g., \cite[formula~(6.7)]{Constantin-Foias} and~\cite[Exercise~II.2.9]{Galdi})
\begin{equation}
\label{eq:parti_ineq}
\left\|v\right\|_{L^{{2d}/{(d-1)}}(\Omega)} \le \textcolor{black}{c_{1}}\left\|v\right\|_0^{1/2}
\left\|  v\right\|_1^{1/2},\qquad \forall v\in H^1(\Omega),
\end{equation}
\textcolor{black}{(where, for simplicity, by enlarging the constants if necessary, we may take the constant~$c_1$ in~(\ref{eq:parti_ineq}) equal to~$c_s$ in~\eqref{sob1} for $s=1$)}
and~Agmon's inequality
\begin{equation}
\label{eq:agmon}
\left\|v \right\|_\infty\le c_{\mathrm{A}} \left\|v\right\|_{d-2}^{1/2} \left\|v\right\|_2^{1/2},\qquad d=2,3,
\qquad \forall v\in H^2(\Omega).
\end{equation}
The case $d=2$ is a direct consequence of~\cite[Theorem 3.9]{Agmon}.
For $d=3$, a proof for domains of class~$C^2$ can be found in~\cite[Lemma~4.10]{Constantin-Foias}.

We will make use of Poincar\'e's inequality,
\begin{equation}
\label{Poin}
\left\| v\right\|_0\le c_P|\Omega |^{1/d} \|\nabla v\|_0 ,\qquad \forall v\in H^1_0(\Omega),
\end{equation}
where the constant~$c_P$ can be taken $c_P\le \sqrt{2}/2$.  If we denote
by $
\hat c_P=1+c_P^2$,  then from~(\ref{Poin}) it follows that
\begin{equation}\label{Poin2}
\left\| v\right\|_1\le (\hat c_P)^{1/2} \|\nabla v\|_0,\qquad \forall v\in H^1_0(\Omega).
\end{equation}

The constants~$c_s$, $c_1$,
$c_A$ and~$c_P$ in the inequalities above are all scale-invariant. This will also be the case of all constants in the present paper unless explicitly stated otherwise.
We will also use the well-known
property, see \cite[Lemma 3.179]{Volker_libro}
\begin{equation}
\label{eq:div}
\left\| \nabla\cdot \bv\right\|_0\le \left\|\nabla \bv\right\|_0,\quad \bv \in
H^1_0(\Omega)^d.
\end{equation}

Let $\mathcal{T}_{h}=(\tau_j^h,\phi_{j}^{h})_{j \in J_{h}}$, $h>0$ be a family of partitions of suitable domains $\Omega_h$, where $h$ denotes the maximum diameter of the elements $\tau_j^h\in \mathcal{T}_{h}$, and $\phi_j^h$ are the mappings from the reference simplex $\tau_0$ onto $\tau_j^h$.
We shall assume that the partitions are shape-regular and quasi-uniform. Let $r \geq 2$, we consider the finite-element spaces
\begin{eqnarray*}
S_{h,r}&=&\left\{ \chi_{h} \in \mathcal{C}\left(\overline{\Om}_{h}\right) \,  \big|
\, {\chi_{h}}_{|{\tau_{j}^{h}}}
\circ \phi^{h}_{j} \, \in \, P^{r-1}(\tau_{0})  \right\} \subset H^{1}(\Om_{h}),
\nonumber\\
{S}_{h,r}^0&=& S_{h,r}\cap H^{1}_{0}(\Om_{h}),
\end{eqnarray*}
where $P^{r-1}(\tau_{0})$ denotes the space of polynomials of degree at most $r-1$ on $\tau_{0}$. For $r=1$,
$S_{h,1}$ stands for the space of piecewise constants.

When $\Omega$ has polygonal or polyhedral boundary $\Omega_h=\Omega$ and mappings~$\phi_j^h$ from the reference simplex are affine.
When $\Omega$ has a smooth boundary, and for the purpose of analysis, we will assume that $\Omega_h$ exactly matches $\Omega$, as it is done
for example in~\cite{chenSiam}, \cite{Schatz98}. At a price of a more involved
analysis,  though, the discrepancies between $\Omega_h$ and
$\Omega$ can also be included in the analysis (see, e.g., \cite{Ay_Gar_Nov}, \cite{Schatz-Whalbin}).

We shall denote by $(X_{h,r}, Q_{h,r-1})$ the MFE pair known as Hood--Taylor elements \cite{BF,hood0} when $r\ge 3$, where
\begin{eqnarray*}
X_{h,r}=\left({S}_{h,r}^0\right)^{d},\quad
Q_{h,r-1}=S_{h,r-1}\cap L^2(\Om_{h})/{\mathbb R},\quad r
\ge 3,
\end{eqnarray*}
and, when $r=2$, the MFE pair known as the mini-element~\cite{Brezzi-Fortin91} where $Q_{h,1}=S_{h,2}\cap L^2(\Om_{h})/{\mathbb R}$, and $X_{h,2}=({S}_{h,2}^0)^{d}\oplus{\mathbb B}_h$. Here, ${\mathbb B}_h$ is spanned by the bubble functions $\bbu_\tau$, $\tau\in\mathcal{T}_h$, defined by $\bbu_\tau(x)=(d+1)^{d+1}\lambda_1(x)\cdots
\lambda_{d+1}(x)$,  if~$x\in \tau$ and 0 elsewhere, where $\lambda_1(x),\ldots,\lambda_{d+1}(x)$ denote the barycentric coordinates of~$x$. All these elements
satisfy a uniform inf-sup condition (see \cite{BF}), that is, for a constant $\beta_{\rm is}>0$ independent of the mesh size $h$ the following inequality holds
\begin{equation}\label{lbbh}
 \inf_{q_{h}\in Q_{h,r-1}}\sup_{v_{h}\in X_{h,r}}
\frac{(q_{h},\nabla \cdot v_{h})}{\|v_{h}\|_{1}
\|q_{h}\|_{L^2/{\mathbb R}}} \geq \beta_{\rm{is}}.
\end{equation}

To approximate the velocity 
we consider the discrete divergence-free space
\begin{eqnarray*}
V_{h,r}=X_{h,r}\cap \left\{ \chi_{h} \in H^{1}_{0}(\Om_{h})^d \mid
(q_{h}, \nabla\cdot\chi_{h}) =0  \quad\forall q_{h} \in Q_{h,r-1}
\right\}.
\end{eqnarray*}
For each fixed time $t\in[0,T]$, notice that the solution $(u,p)$ of \eqref{NS} is also
the solution of a Stokes problem with right-hand side $\bff-\bu_t-(\bu\cdot\nabla )\bu$. We
will denote by $(\bs_h,q_h)\in(X_{h,r},Q_{h,r-1}),$ its  MFE approximation,
solution of
\begin{eqnarray}
\nu(\nabla \bs_h,\nabla \bvar_h)-(q_h,\nabla \cdot \bvar_h)
&=&\nu(\nabla u,\nabla \bvar_h)
-(p,\nabla\cdot \bvar_h)\nonumber\\
&=&(\bff-\bu_t-(\bu\cdot \nabla \bu),\bvar_h)\quad \forall
\bvar_h\in X_{h,r},
\label{stokesnew}\\
(\nabla \cdot \bs_h,\psi_h)&=&0 \quad \forall \psi_h\in Q_{h,r-1}.\nonumber
\end{eqnarray}
Notice that $\bs_h=S_{h}(\bu) : V \rightarrow V_{h,r}$ is the discrete Stokes
projection of the solution $(\bu,p)$ of \eqref{NS} (see \cite{heyran0}), and
it satisfies satisfies that
for all $\bvar_{h} \in V_{h,r}$,
\begin{eqnarray*}
 \nu (\nabla S_h(\bu) , \nabla \bvar_{h} )=\nu( \nabla \bu , \nabla \bvar_{h}) -
 (p, \nabla \cdot \bvar_{h})=(\bff-\bu_{t}-(\bu\cdot \nabla )\bu , \bvar_{h}).
\end{eqnarray*}
The following bound holds (see e.g., \cite{heyran2}):
\begin{equation}
\|\bu-\bs_h\|_0+h\|\bu-\bs_h\|_1\le CN_j(\bu,p) h^j,\qquad
1\le j\le r,
\label{stokespro}
\end{equation}
where here and in the sequel, for $\bv\in V\cap H^j(\Omega)^d$ and~$q\in L^2_0(\Omega)\cap H^{j-1}(\Omega)$ we denote
\begin{equation}
\label{eq:N_j}
 N_j(\bv,q) = \|\bv\|_j + \nu^{-1} \|q\|_{H^{j-1}/{\mathbb R}}, \qquad j\ge 1,
\end{equation}
and, when $\bv$ and~$q$ depend on~$t$
\begin{equation}
\label{eq:barraN}
\overline{N}_j(\bv,q)=\sup_{\tau\ge 0}N_r(\bv(\tau),q(\tau)),\
M_j(\bv)=\sup_{\tau\ge 0} \| \bv(\tau)\|_j,\ M_j(p)=\sup_{\tau\ge 0} \| p(\tau)\|_{H^{j}/{\mathbb R}}.
\end{equation}
Assuming that $\Omega$ is of class ${\cal C}^m$,
with $m \ge 3$, and using
standard duality arguments and~\eqref{stokespro}, one obtains
\begin{equation}\label{eq:stokes_menos1}
\| \bu-\bs_{h}\|_{-s}  \leq CN_r(\bu,p) h^{r+s} ,
\qquad 0 \leq s\leq \min(r-2, 1).
\end{equation}
We also have the following bounds~\cite[Lemma~3.7]{Ours}
\begin{align}
\label{cota_sh_inf}
\|\bs_h\|_\infty  & \le D_0\left((\|\bu\|_{d-2}\|\bu\|_2)^{1/2} +\bigl(N_1(\bu,p)N_{d-1}(\bu,p)\bigr)^{1/2}\right)
\\
\label{la_cota}
\|\nabla\bs_h\|_{L^{2d/(d-1)}} & \le
 D_0\bigl(N_1(\bu,p)N_2(\bu,p)\bigr)^{1/2},
\end{align}
where the constant~$D_0$ does not depend on~$\nu$.
And, arguing as in~\cite[Lemma~3.7]{Ours} one can also prove
\begin{align}
\label{cota_sh_inf_1}
\|\nabla \bs_h\|_\infty  & \le D_0\left((\|\bu\|_{d-1}\|\bu\|_3)^{1/2} +\bigl(N_2(\bu,p)N_{d}(\bu,p)\bigr)^{1/2}\right).
\end{align}
We also consider a modified Stokes projection that was introduced in \cite{grad-div1} and that we denote by $\bs_h^m:V\rightarrow V_{h,r}$
satisfying
\begin{eqnarray}\label{stokespro_mod_def}
\nu(\nabla \bs_h^m,\nabla \bvar_h)=(\bff-\bu_{t}-(\bu\cdot \nabla )\bu-\nabla p , \bvar_{h}), \quad \forall \, \,
 \bvar_{h} \in V_{h,r},
\end{eqnarray}
and the following error bound, see \cite{grad-div1}:
\begin{equation}
\|\bu-\bs_h^m\|_0+h\|\bu-\bs_h^m\|_1\le C\|\bu\|_j h^j,\qquad
1\le j\le r.
\label{stokespro_mod}
\end{equation}
From \cite{chenSiam}, we also have
\begin{align}
\|\nabla (\bu-\bs_h^m)\|_\infty\le C\|\nabla \bu\|_\infty \label{cotainfty1},
\end{align}
where $C$ does not depend on $\nu$.
Also we have the following bounds~\cite[Lemma~3.8]{Ours}
\begin{align}
\label{cota_sh_inf_mu}
\|\bs_h^m\|_\infty  & \le D_1(\|\bu\|_{d-2}\|\bu\|_2)^{1/2},
\\
\label{la_cota_mu}
\|\nabla\bs_h^m\|_{L^{2d/(d-1)}} & \le
 D_1\bigl(\|\bu\|_1\|\bu\|_2\bigr)^{1/2},
\end{align}
where the constant~$D_1$ is independent of~$\nu$.

We will denote by $\pi_h p$ the $L^2$ projection of the pressure $p$ onto $Q_{h,r-1}$. It is well-known that
\begin{equation}\label{eq:L2p}
\|p-\pi_h p\|_0\le C h^{j-1}\|p\|_{H^{j-1}/{\mathbb R}},\qquad 1\le j\le r.
\end{equation}



We will assume that the interpolation operator $I_H$ is stable in $L^2$, that is,
\begin{eqnarray}\label{eq:L^2inter}
\|I_H \bu\|_0\le c_0\|\bu\|_0,\quad \forall \bu\in L^2(\Omega)^d,
\end{eqnarray}
and that it satisfies the following approximation property,
\begin{eqnarray}\label{eq:cotainter}
\|\bu-I_H\bu\|_0\le c_I H\|\nabla \bu\|_0,\quad \forall \bu\in H_0^1(\Omega)^d.
\end{eqnarray}
The Bernardi--Girault~\cite{Ber_Gir}, Girault--Lions~\cite{Girault-Lions-2001}, or the Scott--Zhang~\cite{Scott-Z} interpolation operators
satisfy
\eqref{eq:cotainter} and~\eqref{eq:L^2inter}. Notice that the interpolation can be
on
piecewise constants.

Finally, we will denote by $I_h^{La} \bu\in X_{h,r}$  the Lagrange interpolant of a continuous function $\bu$. In Subsection~\ref{sub_la}
we consider the case in which $I_H=I_H^{La}$ for which bounds \eqref{eq:L^2inter} and \eqref{eq:cotainter} do not hold.

\section{Fully discrete schemes}
\label{Se:main}
We consider the following method to approximate \eqref{eq:mod_NS}. For $n\ge 1$  we define $(\bu_h^{(n)},p_h^{(n)})\in X_{h,r}\times Q_{h,r-1}$ satisfying
for all $(\bvar_h,\psi_h)\in X_{h,r}\times Q_{h,r-1}$
\begin{align}\label{eq:method}
(d_t\bu_h^{(n)},\bvar_h)
+\nu(\nabla \bu_h^{(n)},&\nabla \bvar_h)+b_h(\bu_h^{(n)},\bu_h^{(n)},\bvar_h)+\mu(\nabla \cdot \bu_h^{(n)},\nabla \cdot \bvar_h)
\nonumber\\
+(\nabla p_h^{(n)},\bvar_h)&=(\bff^{(n)},\bvar_h)
-{\beta(I_H(\bu_h^{(n)})-I_H(\bu^{(n)}),I_H\bvar_h)},\nonumber\\
(\nabla \cdot \bu_h^{(n)},\psi_h)&=0.
\end{align}
In \eqref{eq:method} $d_t\bu_h^{(n)}$ is a finite difference approximation to the time derivative at time $t_n=n\Delta t$, where $\Delta t$
is the time step. Also, $\mu$ is a stabilization parameter that can be zero in case we do not stabilize the divergence or different from zero in case we add grad-div stabilization and $b_h(\cdot,\cdot,\cdot)$  is defined in the following way
$$
b_{h}(\bu_{h},\bv_{h},\bvar_{h}) =((\bu_{h}\cdot \nabla ) \bv_{h}, \bvar_{h})+ \frac{1}{2}( \nabla \cdot (\bu_{h})\bv_{h},\bvar_{h}),
\quad \, \forall \, \bu_{h}, \bv_{h}, \bvar_{h} \in X_{h,r}.
$$
It is straightforward to verify that $b_h$ enjoys the skew-symmetry property
\begin{equation}\label{skew}
b_h(\bu,\bv,\bw)=-b_h(\bu,\bw,\bv) \qquad \forall \, \bu, \bv, \bw\in H_0^1(\Omega)^d.
\end{equation}
Let us observe that taking $\bvar_h\in V_{h,r}$ from \eqref{eq:method} we get
\begin{align}\label{eq:method2}
(d_t\bu_h^{(n)},\bvar_h)
+\nu(\nabla \bu_h^{(n)},\nabla &\bvar_h)+b_h(\bu_h^{(n)},\bu_h^{(n)},\bvar_h)+\mu(\nabla \cdot \bu_h^{(n)},\nabla \cdot\bvar_h)={}\nonumber\\
&(\bff^{(n)},\bvar_h)-\beta(I_H(\bu_h^{(n)})-I_H(\bu^{(n)}),I_H\bvar_h).
\end{align}

For the analysis below, we  introduce the values~$\overline\mu$ and
$\overline k$, defined as follows
\begin{equation}
\label{eq:mubarra}
\overline \mu=\left\{\begin{array}{lcl} 0,&\quad& \hbox{\rm if~$\mu=0$},\\ 1,&& \hbox{\rm otherwise},\end{array}\right.
\qquad\qquad
\overline k=\left\{\begin{array}{lcl} 0,&\quad& \hbox{\rm if~$\mu=0$},\\ 1/\mu,&& \hbox{\rm otherwise}.\end{array}\right.
\end{equation}

\begin{lema}\label{le:nonlinb} The following bound holds for $\bv_h,\hat \bv_h,\bw_h,
\hat  \bw_h\in V_h$, and
$\varepsilon,\delta>0$,
\begin{align*}
|b_h(\hat  \bv_h,\bv_h,\be_h)
-b_h(\hat  \bw_h,\bw_h,\be_h)| \le& \frac{1}{\delta}\hat  L(\bw_h,\varepsilon)\|\be_h\|_0^2\nonumber\\
&{}+
\delta\left(\frac{\|\nabla\bw\|_{\infty}}{2}\|\hat  \be_h\|_0^2+\frac{\varepsilon}{4}\|\nabla\cdot\hat  \be_h\|_0^2\right),
\end{align*}
where $\be_h=\bv_h-\bw_h$, $\hat  \be_h=\hat  \bv_h-\hat  \bw_h$ and
\begin{equation}
\label{eq:L1b}
\hat   L(\bw_h,\varepsilon)=\left(\frac{\|\nabla \bw_h\|_\infty}{2}+ \frac{\|\bw_h\|_\infty^2}{4\varepsilon}\right),
\end{equation}
\end{lema}
\begin{proof}
Adding $\pm b_h(\hat  \bv_h,\bw_h,\be_h)$ and applying the
skew-symmetry property \eqref{skew} we have
\begin{align*}
|b_h(\hat  \bv_h,\bv_h,\be_h)
-b_h(\hat  \bw_h,\bw_h&,\be_h)|=|b_h(\hat  \be_h,\bw_h,\be_h)|\le \|\nabla \bw_h\|_{\infty}\|\hat  \be_h\|_0\|\be_h\|_0\\
{}+\frac{1}{2}\|\nabla \cdot \hat  \be_h\|_0\|\bw_h\|_\infty\|\be_h\|_0&\le
\frac{ \|\nabla \bw_h\|_{\infty}}{2\delta }\|\be_h\|_0^2+\delta \frac{ \|\nabla \bw_h\|_{\infty}}{2 }\|\hat  \be_h\|_0^2 \\
&\quad{}+\frac{\|\bw_h\|_\infty^2}{4\delta\epsilon}\|\be_h\|_0^2+\frac{\delta\epsilon}{4}\|\nabla \cdot \hat \be_h\|_0^2.
\nonumber
\end{align*}
\end{proof}

The proof of the following lemma is similar to the proof of \cite[Lemma 3.1]{Ours}.
\begin{lema}\label{lema_general}
Let $(\bu_h^{(n)})_{n=0}^\infty$ be the finite element approximation defined in
\eqref{eq:method2} and let $(\bw_h^{(n)})_{n=0}^\infty$, $(\btau_h^{(n)})_{n=0}^\infty$, $(\theta_h^{(n)})_{n=1}^\infty$ in  $V_{h,r}$ be sequences satisfying
\begin{align}\label{eq:wh}
(d_t \bw_h^{(n)},\bvar_h)+\nu(\nabla \bw_h^{(n)},&\nabla \bvar_h)+b_h(\bw_h^{(n)},\bw_h^{(n)},\bvar_h)+\mu(\nabla \cdot \bw_h^{(n)},\nabla \cdot \bvar_h)
={}\nonumber\\
&(\bff^{(n)},\bvar_h)+(\btau_h^{(n)},\bvar_h)+\overline \mu (\theta_h^{(n)},\nabla \cdot \bvar_h),
\end{align}
Assume that the quantity~$L$
defined
in~\eqref{eq:L_muno0}, below,
 is bounded. Then, if $\beta\ge 8L$ and $H$ satisfies condition~\eqref{eq:as2}, below, the following bounds hold for~$\be_h^{(n)}=\bu_h^{(n)}-\bw_h^{(n)}$,
 \begin{eqnarray}\label{eq:muboth_fully}
(d_t\be_h^{(n)},\be_h^{(n)})+\frac{\gamma}{2} \|\be_h^{(n)}\|_0^2
+\frac{\mu}{2}\|\nabla \cdot \be_h^{(n)}\|_0^2
\le \overline k \|\theta_h^{(n)}\|_0^2+\frac{\beta }{2}c_0^2\|\bu^{(n)}-\bw_h^{(n)}\|_0^2
 \nonumber \\
 +\left((1-\overline \mu)\frac{\hat c_P}{\nu}+\frac{{\overline \mu}}{2L}\right)\|\btau_h^{(n)}\|_{-1+\overline \mu}^2,
\end{eqnarray}
where, $\overline\mu$ and~$\overline k$ are defined in~\eqref{eq:mubarra}, and  $\gamma$ is defined in \eqref{eq:gamma} below.
\end{lema}
\begin{proof} Subtracting \eqref{eq:wh} from \eqref{eq:method2} we get the error equation
\begin{align}\label{eq:error1}
&(d_t\be_h^{(n)},\bvar_h)+\nu(\nabla\be_h^{(n)},\nabla\bvar_h)+\beta(I_H \be_h^{(n)},I_H\bvar_h)\\
&{}+b_h(\bu_h^{(n)},\bu_h^{(n)},\bvar_h)
-b_h(\bw_h^{(n)},\bw_h^{(n)},\bvar_h)+\mu(\nabla \cdot \be_h^{(n)},\nabla \cdot \bvar_h)=\nonumber\\
&
\beta(I_H \bu^{(n)}-I_H \bw_h^{(n)},I_H \bvar_h)+(\btau_h^{(n)},\bvar_h)+\overline \mu(\theta_h^{(n)},\nabla \cdot \bvar_h),
\nonumber
\end{align}
for all $\bvar_h\in V_{h,r}$
Taking $\bvar_h=\be_h^{(n)}$ in \eqref{eq:error1} we get
\begin{align}\label{eq:error2}
(d_t\be_h^{(n)},\be_h^{(n)})+\nu\|\nabla \be_h^{(n)}\|_0^2+\beta \|I_H\be_h^{(n)}\|_0^2+&\mu\|\nabla \cdot \be_h^{(n)}\|_0^2
\\
\le |b_h(\bu_h^{(n)},\bu_h^{(n)},\be_h^{(n)}-b_h(\bw_h^{(n)},\bw_h^{(n)},\be_h&^{(n)})|
 +\beta|(I_H \bu^{(n)}-I_H \bw_h^{(n)},I_H\be_h^{(n)})|\nonumber\\
&\quad+|(\btau_h^{(n)},\be_h^{(n)})|+|\overline \mu(\theta_h^{(n)},\nabla \cdot \be_h^{(n)})|.\nonumber
\end{align}
We will bound the terms on the right-hand side of \eqref{eq:error2}. For the nonlinear term and the truncation errors
we argue differently  depending on whether $\mu=0$ or $\mu>0$.

If $\mu=0$, applying Lemma~\ref{le:nonlinb} with $\hat\bv_h=\bv_h=\bu_h^{(n)}$,
$\hat\bw_h=\bw_h=\bw_h^{(n)}$, $\epsilon=\nu$ and~$\delta=1$,  we have
\begin{align}\label{eq:error3}
&|b_h(\bu_h^{(n)},\bu_h^{(n)},\be_h^{(n)})
-b_h(\bw_h^{(n)},\bw_h^{(n)},\be_h^{(n)})| \nonumber\\
&\qquad {}\le \left(\hat L(\bw_h^{(n)},\nu)+\frac{1}{2}\|\nabla\bw_h^{(n)}\|_{\infty}\right)
\|\be_h^{(n)}\|_0^2+\frac{\nu}{4}\|\nabla\cdot\be_h^{(n)}\|_0^2,
\nonumber\\
&\qquad {}\le\left(\hat L(\bw_h^{(n)},\nu)+\frac{1}{2}\|\nabla\bw_h^{(n)}\|_{\infty}\right)
\|\be_h^{(n)}\|_0^2+\frac{\nu}{4}\|\nabla \be_h^{(n)}\|_0^2
\end{align}
where in the last inequality we have applied~\eqref{eq:div}.
For the term $|(\btau_h^{(n)},\be_h^{(n)})|$ when $\mu=0$, using \eqref{Poin2} we get
\begin{align}\label{eq:error6}
|(\btau_h^{(n)},\be_h^{(n)})|&\le \|\btau_h^{(n)}\|_{-1}\|\be_h^{(n)}\|_1\le \textcolor{black}{(\hat c_P)^{1/2}}\|\btau_h^{(n)}\|_{-1}\|\nabla \be_h^{(n)}\|_0
\nonumber\\
&{}\le \frac{\textcolor{black}{\hat c_P}}{\nu}\|\btau_h^{(n)}\|_{-1}^2+\frac{\nu}{4}\|\nabla \be_h^{(n)}\|_0^2.
\end{align}

When $\mu> 0$, we bound the nonlinear term by applying
Lemma~\ref{le:nonlinb} with $\hat\bv_h=\bv_h=\bu_h^{(n)}$,
$\hat\bw_h=\bw_h=\bw_h^{(n)}$, $\epsilon=\mu$ and~$\delta=1$, that is
\begin{align}\label{eq:error3_muno0}
&|b_h(\bu_h^{(n)},\bu_h^{(n)},\be_h^{(n)})
-b_h(\bw_h^{(n)},\bw_h^{(n)},\be_h^{(n)})|\nonumber\\
&\qquad{}\le \left(\hat L(\bw_h^{(n)},\mu)+\frac{1}{2}\|\nabla\bw_h^{(n)}\|_{\infty}\right)\|\be_h^{(n)}\|_0^2+\frac{\mu}{4}\|\nabla\cdot\be_h^{(n)}\|_0^2,
\end{align}
In the sequel we denote
\begin{eqnarray}\label{eq:L}
L&=&\max_{n \ge 0}\left(\hat L(\bw_h^{(n)},\nu)+\frac{1}{2}\|\nabla\bw_h^{(n)}\|_{\infty}\right)\quad {\rm if}\quad \mu=0,
\\
\label{eq:L_muno0}
L&=&\max_{n\ge 0}\left(2\hat L(\bw_h^{(n)},\mu)+\|\nabla\bw_h^{(n)}\|_{\infty}\right) ,\hfill\quad {\rm if}\quad \mu> 0,
\end{eqnarray}
Observe that in the case~$\mu=0$, the left-hand side of~(\ref{eq:error3}) can be bounded by $L\|\be_h^{(n)}\|_0^2+(\nu/4)\|\nabla\be_h^{(n)}\|_0^2$, and, in the case~$\mu>0$ the left-hand side of~(\ref{eq:error3_muno0}) is bounded by~$(L/2)\|\be_h^{(n)}\|_0^2
+(\mu/4)\|\nabla\cdot\be_h\|_0^2$.

Next, we bound the last two terms of the right-hand side
of~\eqref{eq:error1} when $\mu>0$. We have
\begin{align}\label{eq:error6_muno0}
|(\btau_h^{(n)},\be_h^{(n)})|+|\overline \mu(\theta_h^{(n)},\nabla \cdot \be_h^{(n)})|\le &\frac{1}{2L}\|\btau_h^{(n)}\|_0^2+\frac{L}{2}\|\be_h^{(n)}\|_0^2
\nonumber\\
&{}
+\overline k \|\theta_h^{(n)}\|_0^2+ \frac{\mu}{4}\|\nabla \cdot \be_h^{(n)}\|_0^2,
\end{align}
where $\overline k$ is defined in \eqref{eq:mubarra}.

For the second term on the right-hand side of \eqref{eq:error2} applying \eqref{eq:L^2inter} we get
\begin{eqnarray}\label{eq:error5}
\beta|(I_H \bu^{(n)}-I_H \bw_h^{(n)},I_H\be_h)|&\le &\beta c_0\|\bu^{(n)}-\bw_h^{(n)}\|_0\|I_H\be_h^{(n)}\|_0
\nonumber \\
&\le& \frac{\beta}{2} c_0^2\|\bu^{(n)}-\bw_h^{(n)}\|_0^2
+\frac{\beta}{2}\|I_H\be_h^{(n)}\|_0^2.
\end{eqnarray}
Inserting \eqref{eq:error3},
\eqref{eq:error6}, \eqref{eq:error3_muno0}, \eqref{eq:error6_muno0} and \eqref{eq:error5} into \eqref{eq:error2} we get
\begin{align}\label{eq:otra}
&(d_t\be_h^{(n)},\be_h^{(n)})+(1+\overline  \mu)\frac{\nu}{2}\|\nabla \be_h^{(n)}\|_0^2+\frac{\beta}{2} \|I_H \be_h^{(n)}\|_0^2
+\frac{\mu}{2}\|\nabla \cdot \be_h^{(n)}\|_0^2
\\ &{}\le L\|\be_h\|_0^2+\overline k \|\theta_h^{(n)}\|_0^2+\frac{\beta}{2} c_0^2\|\bu^{(n)}-\bw_h^{(n)}\|_0^2+\left((1-\overline \mu)\frac{\hat c_P}{\nu}+\frac{\overline \mu}{2L}\right)\|\btau_h^{(n)}\|_{-1+\overline \mu}^2.
\nonumber
\end{align}
Now we bound
\begin{eqnarray*}
L\|\be_h^{(n)}\|_0^2\le 2L \|I_H e_h^{(n)}\|_0^2+2L \|(I-I_H)e_h^{(n)}\|_0^2.
\end{eqnarray*}
Assuming that $\beta\ge 8L$ and taking into account that $1+\overline \mu\ge 1$ we get
\begin{align}
\label{eq:aver}
&(d_t\be_h^{(n)},\be_h^{(n)})+\frac{\nu}{2}\|\nabla \be_h^{(n)}\|_0^2
-2L\|(I-I_H)\be_h^{(n)}\|_0^2+\frac{\beta}{4} \|I_H \be_h^{(n)}\|_0^2
+\frac{\mu}{2}\|\nabla \cdot \be_h^{(n)}\|_0^2
\nonumber\\
& {}\le \overline k \|\theta_h^{(n)}\|_0^2+\frac{\beta }{2}c_0^2\|\bu^{(n)}-\bw_h^{(n)}\|_0^2+\left((1-\overline \mu)\frac{\hat c_P}{\nu}+\frac{{\overline \mu}}{2L}\right)\|\btau_h^{(n)}\|_{-1+\overline \mu}^2.
\end{align}
For the rest of the proof we argue exactly as in the proof of \cite[Lemma 3.1]{Ours}. For the third term on the left-hand side above applying \eqref{eq:cotainter}  and assuming
\begin{equation}
\label{eq:as2}
H\le \frac{\nu^{1/2}}{(8L)^{1/2}c_I}
\end{equation}
we get $
-2L\|(I-I_H)e_h^{(n)}\|_0^2 \ge -\frac{\nu}4\|\nabla e_h^{(n)}\|_0^2$.
Therefore, for the last three terms on the left-hand side of~(\ref{eq:aver}) we have
\begin{equation}\label{eq:aver2}
\frac{\nu}2\|\nabla \be_h^{(n)}\|_0^2+\frac{\beta}{4} \|I_H \be_h^{(n)}\|_0^2
-2L\|(I-I_H)\be_h^{(n)}\|_0^2\ge \frac{\nu}{4}\|\nabla \be_h^{(n)}\|_0^2+\frac{\beta}{4} \|I_H \be_h^{(n)}\|_0^2.
\end{equation}
Now, applying \eqref{eq:cotainter} again to bound  the right-hand side above we have that
\begin{eqnarray}\label{eq:apply(21)b}
 \frac{\nu}{4}\|\nabla \be_h^{(n)}\|_0^2+\frac{\beta}{4} \|I_H \be_h^{(n)}\|_0^2
 \ge  {\gamma} (\|I_H\be_h^{(n)}\|_0^2+\|(I-I_H)\be_h^{(n)}\|_0^2),
\end{eqnarray}
where
\begin{equation}
\label{eq:gamma}
\gamma=\min\left\{\frac{\nu}{4}c_I^{-2}H^{-2},\frac{\beta}{4}\right\}.
\end{equation}
Finally, since
$
 {\gamma} (\|I_H\be_h^{(n)}\|_0^2+\|(I-I_H)\be_h^{(n)}\|_0^2)\ge (\gamma/2) \|\be_h^{(n)}\|_0^2
$,
from~\eqref{eq:aver}, \eqref{eq:aver2} and~\eqref{eq:apply(21)b}
we conclude~\eqref{eq:muboth_fully}.
\end{proof}
\begin{lema}\label{larios} Assume that the constants $\alpha$ and $B$ satisfy $0<\alpha<1$ and $B\ge 0$.
Then if the sequence of real numbers~$(a_n)_{n=0}^N$ satisfies
$$
a_{n}\le \alpha a_{n-1} +B, \qquad n=1,\ldots,N
$$
we have that
$$
a_n\le \alpha^{j} a_{n-j} + \frac{B}{1-\alpha},\qquad 0\le j\le n\le N.
$$
\end{lema}
\begin{proof} See e.g., \cite{Larios_et_al}.\end{proof}
The following result is taken from~\cite{Ours}.
\begin{lema}\label{le:est-1} The following bounds hold
\begin{eqnarray}
\label{eq:K_0}
\sup_{\|\bvar\|_1=1}| b_h(\bu,\bu,\bvar)-b_h(\bs_h,\bs_h,\bvar)| &\le& K_0(\bu,p,|\Omega|) \|\bu-\bs_h\|_0,
\\
\sup_{\|\bvar\|_0=0}| b_h(\bu,\bu,\bvar)-b_h(\bs_h^m,\bs_h^m,\bvar)| &\le& K_1(\bu,|\Omega|) \|\bu-\bs_h^m\|_1,
\label{eq:K_1}
\end{eqnarray}
where
\begin{equation}
\label{Cup}
K_0(\bu,p,|\Omega|)=C\Bigl( K_1(\bu,|\Omega|)
+\overline N_1(\bu,p)^{1/2}\bigl(\overline N_{d-1}(\bu,p)+
|\Omega|^{(3-d)/d}\overline N_{2}(\bu,p)\bigr)^{1/2}\Bigr),
\end{equation}
\begin{equation}
\label{Cu}
K_1(\bu,|\Omega|) =C\Bigl((M_{d-2}(\bu)_{d-2}M_2(\bu))^{1/2} + |\Omega|^{(3-d)/(2d)}(M_1(\bu)M_2(\bu))^{1/2}\Bigl),
\end{equation}
and $\overline N_j(\bu,p)$ and~$M_j(\bu)$ are the quantities in~\eqref{eq:barraN}.
\end{lema}

We will apply Lemma~\ref{lema_general} taking
$\bw_h^{(n)}=\bs_h(t_n)$, when $\mu=0$, where~$\bs_h$ satisfies~\eqref{stokesnew}, and, when $\mu>0$ taking
 $\bw_h^{(n)}=\bs_h^m(t_n)$, where~$\bs_h^m$ satisfies~\eqref{stokespro_mod_def}.
 Then, it is easy to check that $\bw_h^{(n)} $ satisfies~\eqref{eq:wh} where
\begin{equation}\label{tau1}
(\btau_h^{(n)},\bvar_h)=(\dot\bu^{(n)}-d_t \bw_h^{(n)},\bvar_h)+b_h(\bu^{(n)},\bu^{(n)},\bvar_h)-b_h(\bw_h^{(n)},\bw_h^{(n)},\bvar_h)
\end{equation}
and
\begin{equation}\label{tau2}
(\theta_h^{(n)},\nabla \cdot\bvar_h)=(\pi_h p^{(n)}-p^{(n)},\nabla \cdot \bvar_h)+\mu(\nabla \cdot (\bu^{(n)}-\bw_h^{(n)}),\nabla \cdot \bvar_h).
\end{equation}
Consequently,
we define the following quantities, which are related to the right-hand side
 of~\eqref{eq:muboth_fully}, when $\mu=0$ and when~$\mu>0$, respectively
 \begin{align}
 \label{C0}
 C_0&=\max_{n\ge 0}\biggl(\frac{2\hat c_P}{\nu} \|\btau_h^{(n)} \|_{-1}^2+ \beta c_0^2
 \|\bu^{(n)} - \bw_h^{(n)}\|_0^2\biggr).
 \\
 \label{C1}
C_1&=\max_{n\ge 0}\biggl(\frac{1}{L}\|\btau_h^{(n)}\|_0^2 + \beta c_0^2
 \|\bu^{(n)} - \bw_h^{(n)}\|_0^2 +\frac{2}{\mu} \|\theta_h^{(n)}\|_0^2\biggr),
 \end{align}
 We now estimate the values of~$C_0$ and~$C_1$.
 \begin{lema}\label{le:C0C1_est} For $C_0$ and~$C_1$ defined in~(\ref{C0}) and~(\ref{C1}), respectively,
 the following bounds hold
\begin{align}
\label{eq:cotaC0}
C_0 \le & \frac{4\hat c_P}{\nu}\max_{n\ge 0} \|\bu_t(t_n)-d_t \bw_h^{(n)}\|_{-1}^2
+ \hat C_0^2(r,\beta,\nu,|\Omega|,\bu,p)h^{2r}
\\
\label{eq:cotaC_1}
C_1 \le & \frac{2}{L}\max_{n\ge 0} \|\bu_t(t_n)-d_t \bw_h^{(n)}\|_{0}^2
+\hat C_1^2(r,h,\beta,\mu,|\Omega|,\bu,p) h^{2r-2}
\end{align}
where
\begin{align}
\label{hatC0}
\hat C_0^2(r,\beta,\nu,|\Omega|,\bu,p) = &
C\biggl( \beta+\frac{4\hat c_P}{\nu} K_0^2 (\bu,p,|\Omega|)\biggr)
\overline N_r^2(\bu,p)
\\
\label{hatC1}
\hat C_1^2(r,h,\beta,\mu,|\Omega|,\bu,p) = &C\biggl( \beta h^2+\mu+\frac{2}{L} K_1^2 (\bu,p,|\Omega|)\biggr)
M_r^2(\bu)+\frac{C}{\mu}M^2_{r-1}(p).
\end{align}
 \end{lema}
\begin{proof}  Applying
  \eqref{stokespro} when $\mu=0$ and~\eqref{stokespro_mod} when $\mu>0$ we get
\begin{align*}
\max_{n\ge 0}\|\bu(t_n)-\bw_h^{(n)}\|_0^2&\le C h^{2r}\overline N_r^2(\bu,p),\qquad \mu=0,
\nonumber\\
\max_{n \ge 0}\|\bu(t_n)-\bw_h^{(n)}\|_0^2&\le C h^{2r}M_r^2(\bu),\qquad \mu>0,
\nonumber
\end{align*}
where, recall, $\overline N_r$ and~$M_r$ are defined in~\eqref{eq:barraN}.
Also, applying~Lemma~\ref{le:est-1}  we have
\begin{align*}
\sup_{\|\bvar\|_1=1}|b_h(\bu(t_n),\bu(t_n),\bvar)-b_h(\bw_h^{(n)},\bw_h^{(n)},\bvar)|&\le K_0 (\bu,p,|\Omega|)\|\bu(t_n)-\bw_h^{(n)}\|_0,
\nonumber\\
\sup_{\|\bvar\|_0=1}|b_h(\bu(t_n),\bu(t_n),\bvar)-b_h(\bw_h^{(n)},\bw_h^{(n)},\bvar)|&\le K_1(\bu,|\Omega|)\|\bu(t_n)-\bw_h^{(n)}\|_1,
\end{align*}
so that,
applying \eqref{stokespro} when $\mu=0$ and~\eqref{stokespro_mod} when $\mu>0$
it follows that
\begin{align*}
\|\btau_h^{(n)}\|_{-1}^2 \le & 2\|\bu_t(t_n)-d_t \bw_h^{(n)}\|_{-1}^2
+
2K_0^2 (\bu,p,|\Omega|)\overline N_r^2(\bu,p)h^{2r},\qquad \mu=0,
\nonumber\\
\|\btau_h^{(n)}\|_{0}^2 \le &2\|\bu_t(t_n)-d_t \bw_h^{(n)}\|_{0}^2
+
2K_1^2 (\bu,|\Omega|)M_{r}^2(\bu)h^{2r-2},\qquad \mu>0,
\end{align*}
Similarly, from \eqref{eq:L2p} and \eqref{stokespro_mod} we obtain
\begin{eqnarray*}
\max_{n\ge 1}\|\theta_h^{(n)}\|_0^2\le C h^{2r-2}M^2_{r-1}(p)+C\mu^2 h^{2r-2}M_r^2(\bu),
\end{eqnarray*}
and the proof is finished.\end{proof}

\subsection{Implicit Euler method}
In this case, we set
$$
d_t=d_t^1,
$$
where
\begin{equation}
\label{eq:Euler}
d_t^1 \bu_h^{(n)}=\frac{\bu_h^{(n)}-\bu_h^{(n-1)}}{\Delta t},\qquad n\ge 1.
\end{equation}
For the solution~$\bu$ of~\eqref{NS}, Taylor expansion easily reveals that
\begin{equation}
\label{dt_trunc_eul}
\|\bu_t(t_n)-d_t^1\bu(t_n)\|_j\le C M_{j}(\bu_{tt})(\Delta_t)^2,\qquad j=0,-1,\qquad
n\ge 1.
\end{equation}

For the error $\be_h^{(n)} = \bu_h^{(n)}-\bw_h^{(n)}$,
it is easy to check that the following
relation holds
$$
\Delta t (d_t\be_h^{(n)},\be_h^{(n)}) = \frac{1}{2} \|\be_h^{(n)}\|_0^2 -
 \frac{1}{2} \|\be_h^{(n-1)}\|_0^2 +  \frac{1}{2} \|\be_h^{(n)}-\be_h^{(n-1)}\|_0^2.
 $$

Applying Lemma~\ref{lema_general} we have
\begin{eqnarray}\label{eq:muboth_fully2}
{\|\be_h^{(n)}\|_0^2}-{\|\be_h^{(n-1)}\|_0^2}+{\Delta t\gamma} \|\be_h^{(n)}\|_0^2
&\le& \Delta t C_{\overline \mu},\qquad j=0,1, \quad n\ge 1,
\end{eqnarray}
where $C_0$ and~$C_1$ are defined in~(\ref{C0}--\ref{C1}).
Applying Lemma~\ref{larios} with $$a_n=\|\be_h^{(n)}\|_0^2,\quad 
\quad \alpha=(1+\gamma \Delta t)^{-1},$$
we obtain
\begin{equation}\label{eq:final_euler}
\|\be_h^{(n)}\|_0^2\le \left(\frac{1}{1+\gamma \Delta t}\right)^{n}\|\be_h^{(0)}\|_0^2+\frac{1}{\gamma}C_{\overline\mu},  \quad n\ge 1.
\end{equation}
Notice also that since we are assuming $\beta\ge 8L$ we get
$
{1}/{\gamma}\le \max\left({4}/{\beta},{1}/{(2L)}\right)={1}/{(2L)}$.

Since the value of $C_{\overline\mu}$ is estimated in~Lemma~\ref{le:C0C1_est}
we only have to estimate  $\bu_t(t_n) -d_t\bw_h^{(n)}$.
By writing
\begin{align}
\label{eq:u_t-d_t}
\bu_t(t_n) -d_t\bw_h^{(n)} &= \bu_t(t_n) -d_t\bu(t_n) + d_t (\bu(t_n)-\bw_h^{(n)})
\nonumber\\
&{}=
\bu_t(t_n) -d_t\bu(t_n) +\frac{1}{\Delta t} \int_{t_{n-1}}^{t_n}\partial_t(\bu(\tau)-\bw_h(\tau))\,d\tau,
\end{align}
estimates~\eqref{dt_trunc_eul} and~\eqref{eq:stokes_menos1}
 allow us to write
\begin{equation}\label{eq:final_euler2}
\max_{n\ge 0} \|\bu_t(t_n)-d_t \bw_h^{(n)}\|_{-1}^2 \le M_{-1}^2(\bu_{tt})(\Delta t)^2 +
CN_{r-1}^2(\bu_t,p_t)h^{2r},
\end{equation}
when $\mu=0$ and $r\ge 3$ and~$\Omega$ is of class~${\cal C}^r$. For the mini element or when $\Omega$ is not of class~${\cal C}^r$ we have
\begin{equation}\label{eq:final_euler3}
\max_{n\ge 0} \|\bu_t(t_n)-d_t \bw_h^{(n)}\|_{-1}^2 \le M_{-1}^2(\bu_{tt})(\Delta t)^2 +
C|\Omega|^{2/d}N_{r}^2(\bu_t,p_t)h^{2r}.
\end{equation}
When $\mu>0$, using~\eqref{stokespro_mod} we have
\begin{equation}\label{eq:final_euler4}
\max_{n\ge 0} \|\bu_t(t_n)-d_t \bw_h^{(n)}\|_{0}^2 \le M_{0}^2(\bu_{tt})(\Delta t)^2 +
CM_{r-1}^2(\bu_t)h^{2r-2}.
\end{equation}
Then, from \eqref{eq:final_euler}, \eqref{eq:cotaC0}, \eqref{eq:cotaC_1}, 
\eqref{eq:final_euler2},
\eqref{eq:final_euler3}, \eqref{eq:final_euler4} and applying triangle inequality together with \eqref{stokespro} when $\mu=0$ and~\eqref{stokespro_mod} when $\mu>0$ we conclude the following theorem.
\begin{Theorem}\label{Th:main_muno0}
Assume that  the solution of~\eqref{NS} satisfies that $\bu\in L^\infty(H^s(\Omega)^d)\cap W^{1,\infty}(\Omega)^d$ and
$p\in L^\infty (H^{s-1}(\Omega)/{\mathbb R})$, $s\ge 2$.
Assume also that,  when $\mu=0$,
$\bu_t\in  L^\infty(H^{\max(2,s-1)}(\Omega)^d)$, $p_t\in L^\infty( H^{\max(1,s-2)}(\Omega)
/{\mathbb R})$ and the second derivative satisfies $\bu_{tt}\in L^\infty(H^{-1}(\Omega)^d)$, or, when $\mu>0$,
$\bu_t\in L^\infty(H^{s-1}(\Omega)^d)$ and $\bu_{tt}\in L^\infty(L^2(\Omega)^d)$.
 Let $\bu_h$ be the finite element approximation defined in
\eqref{eq:method2} when $d_t$ is given by~(\ref{eq:Euler}).
Then, if $\beta\ge 8L$ and $H$ satisfies condition~\eqref{eq:as2}
the following bound holds for $n\ge 1$ and $2\le r\le s$,
\begin{align*}
\| \bu(t_n)-\bu_h^{(n)}\|_0 \le& \frac{1}{(1+\gamma \Delta t )^{n/2}}\|\bu_h(0)-\bu(0)\|_0
+\frac{C}{(\nu L)^{1/2}}M_{-1}(\bu_{tt})\Delta t \\
&{}+
\frac{C}{L^{1/2}}\biggl(\hat C_0 +\frac{1}{\nu^{1/2}}|\Omega|^{(1+\hat r-r)/d}\overline N_{\hat r}(\bu_t,p_t)\biggr)h^{r},
\qquad \mu=0.
\\
\| \bu(t_n)-\bu_h^{(n)}\|_0 \le& \frac{1}{(1+\gamma \Delta t )^{n/2}}\|\bu_h(0)-\bu(0)\|_0
+\frac{C}{L} M_0(\bu_{tt})\Delta t\\
&{}+
\frac{C}{L^{1/2}}\biggl(\hat C_1 + \frac{1}{L^{1/2}} M_{r-1}(\bu_t)\biggr)h^{r-1},
\qquad \mu>0.
\end{align*}
where $\gamma$ is defined in \eqref{eq:gamma},
$L$ in~(\ref{eq:L}--\ref{eq:L_muno0}), $\hat C_0$ and~$\hat C_1$
in~\eqref{hatC0} and \eqref{hatC1}, respectively, and $\hat r=r-1$ if $r\ge 3$ and $\Omega$ is of class~${\cal C}^3$ and $\hat r=r$ otherwise.
\end{Theorem}
\begin{remark}\label{re:nu_large} {\rm
For the case $\mu>0$ one can get a bound of size $O(h^r)$ instead of $O(h^{r-1})$ for the spatial component of the error comparing $\bu_h^{(n)}$ instead of
with $\bw_h^{(n)}=\bs_h^m(t_n)$, where~$\bs_h^m$ satisfies~\eqref{stokespro_mod_def}, with $\bw_h^{(n)}=\bs_h(t_n)$,
where~$\bs_h$ satisfies~\eqref{stokesnew} but adding ${}+\mu(\nabla\cdot\bs_h,\nabla\cdot\bvar_h)$ to the left-hand side of the
first equation. However, arguing in this way the constants in the error bounds depend on inverse powers on $\nu$ and then are useful in practice only when $\nu$ is not too small. In the numerical experiments of Section 4 one can observe a rate of convergence $r$ for the method for the bigger values of $\nu$ shown in the figures that decreases to  $r-1$ as $\nu$ diminishes.
This remark applies not only for the error analysis of the implicit Euler method but also for the other methods analyzed below.}
\end{remark}
\begin{remark}{\rm
To prove the existence of solution of the fully discrete scheme~\eqref{eq:method2} for an arbitrary length of the time step one can argue
as in \cite[Theorem 1.2, Chapter II]{temam_exi} with an argument based on the Brouwer's fixed point theorem (see also \cite[Proposition 3.2]{Ibdah_Mondaini_Titi}, \cite[Remark 7.70]{Volker_libro}). To prove uniqueness one can argue as in  \cite[Theorem 3.7]{Ibdah_Mondaini_Titi}   (see also \cite[Remark 7.70]{Volker_libro}).}
\end{remark}
\subsection{Implicit BDF2}
In this case,
we set
$$
d_t=d_t^2,
$$
where
\begin{equation}
\label{eq:BDF2_1}
d_t^2\bu_h^{(n)}=\frac{3\bu_h^{(n)}-4\bu_h^{(n-1)}+\bu_h^{(n-2)}}{2\Delta t},\qquad
n\ge 2.
\end{equation}
and
\begin{equation}
\label{eq:BDF2_2}
d_t^2\bu_h^{(1)} =d_t^1\bu_h^{(1)}= \frac{\bu_h^{(1)}-\bu_h^{(0)}}{\Delta t},
\end{equation}
that is, the first step is performed with the implicit Euler method. As before, Taylor
expansion easily reveals
\begin{equation}
\label{dt_trunc_bdf2}
\|\bu_t(t_n)-d_t^2\bu(t_n)\|_j\le C M_{j}(\bu_{ttt})(\Delta_t)^3,\qquad j=0,-1,\qquad
n\ge 2,
\end{equation}
while for $n=1$, estimate~\eqref{dt_trunc_eul} applies.

For the error
$\be_h^{(n)}=\bu_h^{(n)}-\bw_h^{(n)}$, and $n\ge 2$, it
is well-known that (see, e.g., \cite{HairerII})
\begin{align}\label{eq:enerbdf2}
\Delta t\left(d_t\be_h^{(n)},\be_h^{(n)}\right)=&\|\bE_h^{(n)}\|_G^2
-\|\bE_h^{(n-1)}\|_G^2
\nonumber\\
&{}
+\frac{1}{4}\|\bv_h^{(n)}-2\bv_h^{(n-1)}+\bv_h^{(n-2)}\|_0^2,
\end{align}
where, here and in he sequel we denote $\bE_h^{(n)}=(\be_h^{(n)},\be_h^{(n-1)})$ and
 \begin{align}\label{eq:forG}
\|\bE_h^{(n)}\|_G^2&=
\frac{1}{4}\|\be_h^{(n)}\|_0^2 + \frac{1}{4}\|2\be_h^{(n)} - \be_h^{(n-1)}\|_0^2
\\
&=\frac{5}{4}\|\be_h^{(n)}\|_0^2 -(\be_h^{(n)},\be_h^{(n-1)}) +\frac{1}{4}
\|\be_h^{(n-1)}\|_0^2.\nonumber
\end{align}
Consequently, considering the eigenvalues $\lambda_1\ge \lambda_2>0$
of the matrix
$$
G=\frac{1}{4}\left[\begin{array}{rr} 5& -2\\ -2 & 1\end{array}\right],
$$
we have
\begin{equation}\label{eq:equiv}
\lambda_2\left(\|\be_h^{(n)}\|_0^2 +\|\be_h^{(n-1)}\|_0^2\right)\le
\|\bE_h^{(n)}\|_G^2 \le
\lambda_1\left(\|\be_h^{(n)}\|_0^2 +\|\be_h^{(n-1)}\|_0^2\right).
\end{equation}
And easy calculation shows
\begin{equation}
\label{lambda1}
1\le \lambda_1 =\frac{3+2\sqrt{2}}{4}\le \frac{3}{2}.
\end{equation}
Applying \eqref{eq:muboth_fully} as before and taking into account
the definition of~constants~$C_0$ and~$C_1$ in (\ref{C0}--\ref{C1}) and~\eqref{eq:enerbdf2} we have
\begin{equation}\label{eq:muboth_fully_bdf2}
\|\bE_h^{(n)}\|_G^2
-\|\bE_h^{(n-1)}\|_G^2 +\frac{\gamma}{2}\Delta t \|\be_h^{(n)}\|_0^2
\le \Delta t\frac{C_{\overline \mu}}{2}.
\end{equation}
Now, arguing as in  \cite{Larios_et_al} we add $\pm\gamma \Delta t\|\be_h^{(n-1)}\|_0^2/8$ so that we can write
\begin{align*}
&\|\bE_h^{(n)}\|_G^2
+\frac{3}{8}\gamma\Delta t \|\be_h^{(n)}\|_0^2+\frac{\gamma}{8}\Delta t \left(\|\be_h^{(n)}\|_0^2
+ \|\be_h^{(n-1)}\|_0^2\right)\nonumber\\
&\quad{}\le \|\bE_h^{(n-1)}\|_G^2 +\frac{\gamma}{8}\Delta t \|\be_h^{(n-1)}\|_0^2+\Delta t\frac{C_{\overline \mu}}{2}.
\end{align*}
For the third term on the left-hand side above,
applying \eqref{eq:equiv} we may write
\begin{align}
\label{eq:equiv2}
\frac{\gamma}{8}\Delta t \left(\|\be_h^{(n)}\|_0^2
+ \|\be_h^{(n-1)}\|_0^2\right)& \ge \frac{\gamma}{8\lambda_1}\Delta t
\|\bE_h^{(n)}\|_G^2
 \ge
 \frac{\gamma}{12}\Delta t
\|\bE_h^{(n)}\|_G^2 ,
\end{align}
where in the last inequality we have applied~\eqref{lambda1}
Thus, we have
\begin{align*}
\left(1+\frac{\gamma}{12}\Delta t \right)&\|\bE_h^{(n)}\|_G^2
+\frac{3}{8}\gamma\Delta t \|\be_h^{(n)}\|_0^2
\le \|\bE_h^{(n-1)}\|_G^2 +\frac{\gamma}{8}\Delta t \|\be_h^{(n-1)}\|_0^2+\Delta t\frac{C_{\overline \mu}}{2}.
\end{align*}
Assuming
$$
\frac{3}{8}\gamma\Delta t \ge \left(1+\frac{\gamma}{12}\Delta t \right)
\frac{\gamma}{8}\Delta t,
$$
which holds for
\begin{equation}
\label{eq:restri_Deltat}
\Delta t \le \frac{24}{\gamma},
\end{equation}
we get
\begin{align*}
\left(1+\frac{\gamma}{12}\Delta t \right)&\left(\|\bE_h^{(n)}\|_G^2
+\frac{\gamma}{8}\Delta t \|\be_h^{(n)}\|_0^2\right)
\le \|\bE_h^{(n-1)}\|_G^2 +\frac{\gamma}{8}\Delta t
 \|\be_h^{(n-1)}\|_0^2+\Delta t\frac{C_{\overline\mu}}{2}.
\end{align*}
Applying Lemma~\ref{larios} with $$a_n=\|\bE_h^{(n)}\|_G^2
+\frac{\gamma}{8} \Delta t\|\be_h^{(n)}\|_0^2,\quad \alpha=
\frac{1}{1+\frac{\gamma}{12}\Delta t}$$
and $B=\Delta tC_{\overline \mu}/(2\alpha)$
 we get
\begin{align*}
\|\bE_h^{(n)}\|_G^2
+\frac{\gamma}{8}\Delta t \|\be_h^{(n)}\|_0^2\le& \frac{1}{\left(1+\frac{\gamma}{12} \Delta t\right)^{n-1}}\left(\|\bE_h^{(1)}\|_G^2
+\frac{\gamma}{8} \Delta t\|\be_h^{(1)}\|_0^2\right)
+\frac{6}{\gamma}C_{\overline \mu}.
\end{align*}
Taking into account that from \eqref{eq:forG} we get $\|\bE_h^{(n)}\|_G^2
\ge \|\be_h^{(n)}\|_0^2/4$ and applying \eqref{eq:equiv} again
and that $\lambda_1\le 2$, we have
\begin{equation}
\label{eq:cota_bdf2}
\|\be_h^{(n)}\|_0^2
\le \frac{4}{\left(1+\frac{\gamma}{12} \Delta t\right)^{n-1}}\left(2\left(\|\be_h^{(1)}\|_0^2+\|\be_h^{(0)}\|_0^2\right)
+\frac{\gamma}{8} \Delta t\|\be_h^{(1)}\|_0^2\right)
+\frac{24}{\gamma}C_{\overline \mu}.
\end{equation}
To estimate~$C_{\overline\mu}$  on the right-hand side above, after applying Lemma~\ref{le:C0C1_est}, we are left with the estimation of~$\bu(t_n) -d_t\bw_h^{(n)}$.
Arguing as
in~(\ref{eq:u_t-d_t}),
we may write
$
\bu_t(t_n) -d_t\bw_h^{(n)} = \bu_t(t_n) -d_t\bu(t_n) + d_t (\bu(t_n)-\bw_h^{(n)}),
$ and taking into account that~$d_t=d_t^2$, we  express
\begin{align*}
&d_t (\bu(t_n)-\bw_h^{(n)})
\nonumber\\
&\qquad{}=
\frac{3}{2\Delta t} \int_{t_{n-1}}^{t_n}\partial_t(\bu(\tau)-\bw_h(\tau))\,d\tau-\frac{1}{2\Delta t} \int_{t_{n-2}}^{t_{n-1}}\partial_t(\bu(\tau)-\bw_h(\tau))\,d\tau.
\end{align*}
Thus, using~(\ref{dt_trunc_bdf2}), and~\eqref{eq:stokes_menos1}
 we obtain similar
estimates as~(\ref{eq:final_euler2}--\ref{eq:final_euler4})
but with $M_{j}^2(\bu_{tt})(\Delta t)^2$, $j=-1,0$, replaced by $M_{j}^2(\bu_{ttt})(\Delta t)^4$.
To estimate $\|\be_h^{(1)}\|_0$ on the right-hand side of~(\ref{eq:cota_bdf2}).
we recall that~$\bu_h^{(1)}$ is obtained
by one step of the implicit Euler method, so that we can use~(\ref{eq:muboth_fully2}),
which gives,
\begin{equation}\label{eq:incial1}
{\|\be_h^{(1)}\|_0^2}\le \frac{1}{1+\gamma\Delta t}\left({\|\be_h^{(0)}\|_0^2}
+\Delta t C_{\overline \mu}\right)
\le \frac{1}{1+\frac{\gamma}{12}\Delta t}\left({\|\be_h^{(0)}\|_0^2}
+\Delta t C_{\overline \mu}\right).
\end{equation}
As before, the value of~$C_{\overline\mu}$ is estimated by~Lemma\ref{le:C0C1_est},
\eqref{dt_trunc_eul} and~(\ref{eq:final_euler2}--\ref{eq:final_euler4}). Thus, we conclude with the following result.

\begin{Theorem}\label{th:main_bdf2}
Under the hypotheses of~Theorem~\ref{Th:main_muno0}, assume also that
 $\bu_{ttt}\in L^\infty(H^{-1}(\Omega)^d)$ when $\mu=0$ or
 $\bu_{ttt}\in L^\infty(L^2(\Omega)^d)$, otherwise, and that $\Delta t$ satisfies
 \eqref{eq:restri_Deltat}. Then, for the finite element approximation, solution
 of~\eqref{eq:method2} when $d_t$ is given by~(\ref{eq:BDF2_1}--\ref{eq:BDF2_2}),
 the following bounds hold:
 \begin{align*}
\| \bu(t_n)-\bu_h^{(n)}\|_0 \le& \frac{C}{(1+\frac{\gamma }{12}\Delta t )^{(n-1)/2}}\|\bu_h(0)-\bu(0)\|_0\nonumber\\
&{}
+\frac{C}{(\nu L)^{1/2}}\left(M_{-1}(\bu_{tt})+M_{-1}(\bu_{ttt})\right)(\Delta t)^2 \\
&{}+
\frac{C}{L^{1/2}}\biggl(\hat C_0 +\frac{1}{\nu^{1/2}}|\Omega|^{(1+\hat r-r)/d}\overline N_{\hat r}(\bu_t,p_t)\biggr)h^{r},
\qquad \mu=0.
\\
\| \bu(t_n)-\bu_h^{(n)}\|_0 \le&\frac{C}{(1+\frac{\gamma }{12}\Delta t )^{(n-1)/2}}\|\bu_h(0)-\bu(0)\|_0
\nonumber\\
&{}
+\frac{C}{L} \left(M_0(\bu_{tt})+M_0(\bu_{ttt})\right)(\Delta t)^2\\
&{}+
\frac{C}{L^{1/2}}\biggl(\hat C_1 + \frac{1}{L^{1/2}} M_{r-1}(\bu_t)\biggr)h^{r-1},
\qquad \mu>0.
\end{align*}
where $\gamma$ is defined in \eqref{eq:gamma},
$L$ in~(\ref{eq:L}--\ref{eq:L_muno0}), $\hat C_0$ and~$\hat C_1$
in~\eqref{hatC0} and \eqref{hatC1}, respectively, and $\hat r=r-1$ if $r\ge 3$ and $\Omega$ is of class~${\cal C}^3$ and $\hat r=r$ otherwise.
\end{Theorem}
\begin{remark}{\rm
Let us observe that we are considering a scheme in which the first time step is performed by means of the implicit Euler method and then we can insert
\eqref{eq:incial1} into \eqref{eq:cota_bdf2}. However, in view of \eqref{eq:cota_bdf2}, as pointed out in \cite{Larios_et_al}, the first step could, for example, be initialized to zero and the algorithm still converges to the true solution.}
\end{remark}
\subsection{Semi-implicit BDF2}
Now, we consider a fully discrete approximation satisfying
\begin{align}\label{eq:method2_fully_im}
(d_t \bu_h^{(n)},\bvar_h)
+\nu(\nabla \bu_h^{(n)},&\nabla \bvar_h)+b_h(\tilde \bu_h^{n},\bu_h^{(n)},\bvar_h)+\mu(\nabla \cdot \bu_h^{(n)},\nabla \cdot\bvar_h)={}\nonumber\\
&(\bff^{(n)},\bvar_h)-\beta(I_H(\bu_h^{(n)})-I_H(\bu^{(n)}),I_H\bvar_h),
\end{align}
where $d_t$ is given by~(\ref{eq:BDF2_1}--\ref{eq:BDF2_2}), and
\begin{equation}\label{eq:lautilde}
\tilde \bu_h^{(n)}=2\bu_h^{(n-1)}-\bu_h^{(n-2)}, \qquad n\ge 2,
\end{equation}
and, for simplicity~$\tilde \bu_h^{(1)}= \bu_h^{(0)}$.

Arguing as in the proof
of Lemma~\ref{lema_general} with obvious changes, the following result follows.

\begin{lema}\label{lema_generalb}
Let $(\bu_h^{(n)})_{n=0}^\infty$ be the finite element approximation defined in
\eqref{eq:method2_fully_im} and let $(\bw_h^{(n)})_{n=0}^\infty$,
 $(\tilde \bw_h^{(n)})_{n=0}^\infty$, $(\btau_h^{(n)})_{n=0}^\infty$, $(\theta_h^{(n)})_{n=1}^\infty$ in  $V_{h,r}$ be sequences satisfying
\begin{align}\label{eq:whb}
(d_t \bw_h^{(n)},\bvar_h)+\nu(\nabla \bw_h^{(n)},&\nabla \bvar_h)+b_h(\tilde \bw_h^{(n)},\bw_h^{(n)},\bvar_h)+\mu(\nabla \cdot \bw_h^{(n)},\nabla \cdot \bvar_h)
={}\nonumber\\
&(\bff^{(n)},\bvar_h)+(\btau_h^{(n)},\bvar_h)+\overline \mu (\theta_h^{(n)},\nabla \cdot \bvar_h),
\end{align}
Fix $\delta>0$, and assume that the quantity~$L'$
defined
 in \eqref{eq:Lb}, below, when $\mu=0$, and
in~\eqref{eq:L_muno0b}, below,
when $\mu>0$,
 is bounded. Then, if $\beta\ge 8L'/\delta$ and $H$ satisfies condition
 \begin{equation}
\label{eq:as2b}
H\le \frac{(\nu\delta)^{1/2}}{(8L')^{1/2}c_I}.
\end{equation}
the following bounds hold for~$\be_h^{(n)}=\bu_h^{(n)}-\bw_h^{(n)}$
and~$\tilde \be_h^{(n)}=\tilde \bu_h^{(n)}-\tilde \bw_h^{(n)}$,
 \begin{align}\label{eq:muboth_fullyb}
&(d_t\be_h^{(n)},\be_h^{(n)})+\frac{\gamma}{2} \|\be_h^{(n)}\|_0^2
+\frac{\nu}{4}\|\nabla\be_h^{(n)}\|_0^2
+\frac{3}{4 }\mu\|\nabla \cdot \be_h^{(n)}\|_0^2
 \nonumber \\
    &\qquad {}\le
\delta\left(\frac{1}{2}\|\nabla\bw_h^{(n)} \|_\infty \|\tilde \be_h^{(n)}\|_0^2
+(1-\overline \mu)\frac{\nu}{4}\|\nabla \tilde \be_h^{(n)}\|_0^2 + \frac{\mu}{4}
\|\nabla \cdot\tilde \be_h^{(n)}\|_0^2 \right)\\\
&\qquad {}+ \overline k \|\theta_h^{(n)}\|_0^2+\frac{\beta }{2}c_0^2\|\bu^{(n)}-\bw_h^{(n)}\|_0^2
 +\left((1-\overline \mu)\frac{\hat c_P}{\nu}+\frac{{\overline \mu}}{2L'}\right)\|\btau_h^{(n)}\|_{-1+\overline \mu}^2,\nonumber
\end{align}
where, $\overline\mu$ and~$\overline k$ are defined in~\eqref{eq:mubarra}, and  $\gamma$ is defined in \eqref{eq:gamma}.
\end{lema}
\begin{proof} We follow the proof of Lemma~\ref{lema_general}, that
is, subtracting \eqref{eq:whb} from \eqref{eq:method2_fully_im} and
taking $\varphi_h=\be_h^{(n)}$ we get
\begin{align}\label{eq:error2b}
(d_t\be_h^{(n)},\be_h^{(n)})+\nu\|\nabla \be_h^{(n)}\|_0^2+\beta \|I_H\be_h^{(n)}\|_0^2+&\mu\|\nabla \cdot \be_h^{(n)}\|_0^2
\\
\le|b_h(\tilde \bu_h^{(n)},\bu_h^{(n)},\be_h^{(n)})-b_h(\tilde\bw_h^{(n)},\bw_h^{(n)},\be_h&^{(n)})|
 +\beta|(I_H \bu^{(n)}-I_H \bw_h^{(n)},I_H\be_h^{(n)})|\nonumber\\
&\quad+|(\btau_h^{(n)},\be_h^{(n)})|+|\overline \mu(\theta_h^{(n)},\nabla \cdot \be_h^{(n)})|.\nonumber
\end{align}
To bound the nonlinear terms, we argue as follows.
When $\mu=0$, we apply Lemma~\ref{le:nonlinb}
with $\hat\bv_h=\tilde \bu_h^{(n)}$, $\bv_h=\bu_h^{(n)}$,
$\hat\bw_h=\tilde \bw_h^{(n)}$, $\bw_h=\bw_h^{(n)}$, and~$\epsilon=\nu$ and  we also use~\eqref{eq:div}, so that we have,
\begin{eqnarray}\label{eq:error3b}
&& |b_h(\tilde \bu_h^{(n)},\bu_h^{(n)},\be_h^{(n)})
-b_h(\tilde \bw_h^{(n)},\bw_h^{(n)},\be_h^{(n)})|
\nonumber \\
&&\quad {}\le \frac{1}{\delta}\hat L(\bw_h^{(n)},\nu)\|\be_h^{(n)}\|_0^2
+\frac{\delta}{2}\|\nabla \bw_h^{(n)} \|_\infty\|\tilde \be_h^{(n)}\|_0^2+ \frac{\delta\nu}{4}\|\nabla\cdot\be_h^{(n)}\|_0^2
\nonumber\\
&&\quad {}\le  \frac{1}{\delta}\hat L(\bw_h^{(n)},\nu)\|\be_h^{(n)}\|_0^2
+\frac{\delta}{2}\|\nabla \bw_h^{(n)} \|_\infty\|\tilde \be_h^{(n)}\|_0^2+\frac{\delta\nu}{4}\|\nabla \be_h^{(n)}\|_0^2.
\end{eqnarray}
When $\mu$ is positive, we apply Lemma~\ref{le:nonlinb}
with $\hat\bv_h=\tilde \bu_h^{(n)}$, $\bv_h=\bu_h^{(n)}$,
$\hat\bw_h=\tilde \bw_h^{(n)}$, $\bw_h=\bw_h^{(n)}$, and~~$\epsilon=\mu$,
\begin{align}\label{eq:error3_muno0b}
&|b_h(\tilde \bu_h^{(n)},\bu_h^{(n)},\be_h^{(n)})
-b_h(\tilde\bw_h^{(n)},\bw_h^{(n)},\be_h^{(n)})|
\nonumber \\
&\quad{} \le
\frac{1}{\delta} \hat L(\bw_h^{(n)},\mu)\|\be_h^{(n)}\|_0^2
+\frac{\delta}{2} \|\nabla \bw_h^{(n)}\|_\infty\|\tilde \be_h^{(n)}\|_0^2+\frac{\delta \mu}{4}\|\nabla\cdot\tilde \be_h^{(n)}\|_0^2.
\end{align}
We now set
\begin{eqnarray}\label{eq:Lb}
L'&=&\max_{n \ge 0}\hat L(\bw_h^{(n)},\nu),\quad {\rm if}\quad \mu=0,
\\
\label{eq:L_muno0b}
L'&=&2\max_{n\ge 0}\hat L(\bw_h^{(n)},\mu),\hfill\quad {\rm if}\quad \mu> 0.
\end{eqnarray}
The rest of the terms on the right-hand side of~\eqref{eq:error2b} are bounded as in the proof
of Lemma~\ref{lema_general}, but replacing $L$ by~$L'$. Thus, instead of~(\ref{eq:otra}), we now have
\begin{align}\label{eq:otrab}
&(d_t\be_h^{(n)},\be_h^{(n)})+(3+\overline  \mu)\frac{\nu}{4}\|\nabla \be_h^{(n)}\|_0^2+\frac{\beta}{2} \|I_H \be_h^{(n)}\|_0^2
+\frac{3}{4}\mu\|\nabla \cdot \be_h^{(n)}\|_0^2
\\ &{}\le \frac{1}{\delta} L'\|\be_h\|_0^2 +
\delta\left(\frac{1}{2}\|\nabla\bw_h^{(n)} \|_\infty \|\tilde \be_h^{(n)}\|_0^2
+(1-\overline \mu)\frac{\nu}{4}\|\nabla \tilde \be_h^{(n)}\|_0^2 + \frac{\mu}{4}
\|\nabla \cdot\tilde \be_h^{(n)}\|_0^2 \right)
\nonumber\\
&\quad{}+\overline k \|\theta_h^{(n)}\|_0^2+\frac{\beta}{2} c_0^2\|\bu^{(n)}-\bw_h^{(n)}\|_0^2+\left((1-\overline \mu)\frac{\hat c_P}{\nu}+\frac{\overline \mu}{2L'}\right)\|\btau_h^{(n)}\|_{-1+\overline \mu}^2.
\nonumber
\end{align}
We also bound
$L'\|\be_h^{(n)}\|_0^2\le 2L' \|I_H e_h^{(n)}\|_0^2+2L' \|(I-I_H)e_h^{(n)}\|_0^2$.
Now, since we are assuming that $\beta\ge 8L'/\delta$, so that
$\beta/2-2(L'/\delta)\ge \beta/4$, and taking into account
that $3+\overline \mu\ge 3$
instead of~(\ref{eq:aver}) we get
\begin{align}
\label{eq:averb}
&(d_t\be_h^{(n)},\be_h^{(n)})\!+\!\frac{3}{4}\nu\|\nabla \be_h^{(n)}\|_0^2
-2\frac{L'}{\delta}\|(I-I_H)\be_h^{(n)}\|_0^2\!+\!\frac{\beta}{4} \|I_H \be_h^{(n)}\|_0^2
+\frac{3}{4}\mu|\nabla \cdot \be_h^{(n)}\|_0^2
\nonumber\\
&\qquad {}\le
\delta\left(\frac{1}{2}\|\nabla \bw_h^{(n)} \|_\infty \|\tilde \be_h^{(n)}\|_0^2
+(1-\overline \mu)\frac{\nu}{4}\|\nabla \tilde \be_h^{(n)}\|_0^2 + \frac{\mu}{4}
\|\nabla \cdot\tilde \be_h^{(n)}\|_0^2 \right)
\nonumber\\
&\qquad{}+ \overline k \|\theta_h^{(n)}\|_0^2+\frac{\beta }{2}c_0^2\|\bu^{(n)}-\bw_h^{(n)}\|_0^2+\left((1-\overline \mu)\frac{\hat c_P}{\nu}+\frac{{\overline \mu}}{2L}\right)\|\btau_h^{(n)}\|_{-1+\overline \mu}^2.
\end{align}
Also, since we are now assuming~(\ref{eq:as2b}),
we have
\begin{align}
\label{eq:apply(21)ab}
\frac{\nu}{2}\|\nabla \be_h^{(n)}\|_0^2-2\frac{L'}{\delta}\|(I-I_H)e_h^{(n)}\|_0^2 &\ge
\frac{\nu}4\|\nabla e_h^{(n)}\|_0^2,
\end{align}
Thus, arguing as in the rest of the proof of Lemma~\ref{lema_general},
\eqref{eq:muboth_fullyb} follows.
\end{proof}

As before, we take $\bw_h^{(n)} = \bs_h(t_n)$ if $\mu=0$, and $\bw_h^{(n)} = \bs_h^m(t_n)$, otherwise. Notice that the truncation error~$\btau_h^{(n)}$ now is
\begin{equation}\label{tau1b}
(\btau_h^{(n)},\bvar_h)=(\dot\bu^{(n)}-d_t \bw_h^{(n)},\bvar_h)+b_h(\bu^{(n)},\bu^{(n)},\bvar_h)-b_h(\tilde \bw_h^{(n)},\bw_h^{(n)},\bvar_h)
\end{equation}

We also define
\begin{align}
 \label{C0'}
C_0'&=\max_{n\ge 0}\biggl(\frac{2\hat c_P}{\nu} \|\btau_h^{(n)} \|_{-1}^2+ \beta c_0^2
 \|\bu^{(n)} - \bw_h^{(n)}\|_0^2\biggr).
 \\
 \label{C1'}
C_1'&=\max_{n\ge 0}\biggl(\frac{1}{L'}\|\btau_h^{(n)}\|_0^2 + \beta c_0^2
 \|\bu^{(n)} - \bw_h^{(n)}\|_0^2 +\frac{2}{\mu} \|\theta_h^{(n)}\|_0^2\biggr),
 \end{align}

Thus, applying Lemma~\ref{lema_generalb} and recalling~\eqref{eq:enerbdf2}, we have
 \begin{align*}
&\|\bE_h^{(n)}\|_G^2 - \|\bE_h^{(n-1)}\|_G^2+\frac{\gamma}{2}\Delta t \|\be_h^{(n)}\|_0^2
+\frac{\nu}{4}\Delta t \|\nabla\be_h^{(n)}\|_0^2
+\frac{3}{4 }\Delta t \mu\|\nabla \cdot \be_h^{(n)}\|_0^2
\nonumber \\
    &\quad {}\le
\delta\Delta t \left(\frac{1}{2}\|\nabla \bw_h^{(n)} \|_\infty \|\tilde \be_h^{(n)}\|_0^2
+(1-\overline \mu)\frac{\nu}{4}\|\nabla \tilde \be_h^{(n)}\|_0^2 + \frac{\mu}{4}
\|\nabla \cdot\tilde \be_h^{(n)}\|_0^2 \right)+\frac{\Delta t }{2}C_{\overline\mu}'.
\nonumber
\end{align*}
In view of the defintions of~$\gamma$ in~\eqref{eq:gamma}, $\tilde L$ in~\eqref{eq:L1b} and $L'$ in~(\ref{eq:Lb}--\ref{eq:L_muno0b}),
and the restriction~\eqref{eq:as2b} we have
$$
\frac{1}{2}\|\nabla\bw_h^{(n)} \|_\infty \le L' \le \delta \frac{\gamma}{2}.
$$
Thus, we may write
 \begin{align}\label{eq:semimp2}
&\|\bE_h^{(n)}\|_G^2 - \|\bE_h^{(n-1)}\|_G^2+\frac{\gamma}{2}\Delta t \|\be_h^{(n)}\|_0^2
+\frac{\nu}{4}\Delta t \|\nabla\be_h^{(n)}\|_0^2
+\frac{3}{4 }\Delta t \mu\|\nabla \cdot \be_h^{(n)}\|_0^2
\nonumber \\
    &\quad {}\le
\delta\Delta t \left(\delta\frac{\gamma}{2} \|\tilde \be_h^{(n)}\|_0^2
+(1-\overline \mu)\frac{\nu}{4}\|\nabla \tilde \be_h^{(n)}\|_0^2 + \frac{\mu}{4}
\|\nabla \cdot\tilde \be_h^{(n)}\|_0^2 \right)+\frac{\Delta t }{2}C_{\overline\mu}'.
\end{align}
We add $\pm \Delta t(\gamma /8) \|\be_h^{(n-1)}\|_0^2$ to the left hand side above,
so that recalling~\eqref{eq:equiv2} and
noticing that
$$
\|\bE_h^{(n-1)}\|_G^2 =\frac{1}{4} \|\be_h^{(n-1)}\|_0^2 + \frac{1}{4}\|\tilde \be_h^{(n)}\|_0^2
$$
and
\begin{align*}
\delta\frac{\gamma}{2} \|\tilde \be_h^{(n)}\|_0^2&
+(1-\overline \mu)\frac{\nu}{4}\|\nabla \tilde \be_h^{(n)}\|_0^2 + \frac{\mu}{4}
\|\nabla \cdot\tilde \be_h^{(n)}\|_0^2
\\
&{}\le 4\left( \delta\frac{\gamma}{2} \| \bE_h^{(n-1)}\|_G^2
+(1-\overline \mu)\frac{\nu}{4}\|\nabla \bE_h^{(n-1)}\|_0^2 + \frac{\mu}{4}
\|\nabla \cdot \bE_h^{(n-1)}\|_G^2\right)
\end{align*}
where $\nabla \bE_h^{(n)}=( \nabla \be_h^{(n)},  \nabla \be_h^{(n-1)})$ and
 $\nabla\cdot  \bE_h^{(n)}=( \nabla\cdot  \be_h^{(n)},  \nabla\cdot \be_h^{(n-1)})$,
we may write
 \begin{align}\label{eq:semimp2}
&\left(1+\frac{\gamma}{12}\Delta t\right)\|\bE_h^{(n)}\|_G^2 +
{\Delta t}\left(\frac{3}{8}\gamma \|\be_h^{(n)}\|_0^2+
\frac{\nu}{4}\|\nabla\be_h^{(n)}\|_0^2
+\frac{3}{4}\mu\|\nabla \cdot \be_h^{(n)}\|_0^2\right)
\nonumber \\
    &\quad {}\le \left(1+2\delta^2\gamma\Delta t\right)\|\bE_h^{(n-1)}\|_G^2
 +\Delta t\frac{\gamma}{8}\|\be_h^{(n-1)} \|_0^2
\nonumber    \\
&\quad {}+\delta\Delta t\left(
(1-\overline \mu)\nu \|\nabla \bE_h^{(n-1)}\|_0^2 + \mu
\|\nabla \cdot \bE_h^{(n-1)}\|_G^2\right)
 +\frac{\Delta t }{2}C_{\overline\mu}'.
\end{align}
We treat separately the cases $\mu>0$, and $\mu=0$. For the former, we drop
the term~$\Delta t(\nu/4) \|\nabla\be_h^{(n)}\|_0^2$ on the left hand side
of~\eqref{eq:semimp2} and
for the last term on the left-hand side of~(\ref{eq:semimp2}) we write
\begin{align*}
\frac{3\mu}{4}\|\nabla \cdot \be_h^{(n)}\|_0^2&=\frac{\mu}{2}\|\nabla \cdot \be_h^{(n)}\|_0^2+
\frac{\mu}{4}\left(\|\nabla \cdot \be_h^{(n)}\|_0^2 + \|\nabla \cdot \be_h^{(n-1)}\|_0^2
\right) - \frac{\mu}{4}  \|\nabla \cdot \be_h^{(n-1)}\|_0^2
\nonumber\\
&{}\ge \frac{\mu}{2}\|\nabla \cdot \be_h^{(n)}\| +
\frac{\mu}{6} \|\nabla\cdot\bE_h^{(n)}\|_G^2 -
\frac{\mu}{4}  \|\nabla \cdot \be_h^{(n-1)}\|_0^2
\end{align*}
where, in the last inequality we have argued as in~\eqref{eq:equiv2}.
Thus, from~(\ref{eq:semimp2}) it follows that
 \begin{align}\label{eq:semimp3}
&\left(1+\frac{\gamma}{12}\Delta t\right)\|\bE_h^{(n)}\|_G^2 +
\Delta t \frac{\mu}{6} \|\nabla\cdot\bE_h^{(n)}\|_G^2 +
{\Delta t}\left(\frac{3}{8}\gamma \|\be_h^{(n)}\|_0^2
+\frac{\mu}{2}\|\nabla \cdot \be_h^{(n)}\|_0^2\right)
\nonumber \\
     &\quad {}\le \left(1+2\delta^2\gamma\Delta t\right)\|\bE_h^{(n-1)}\|_G^2
     +\delta \Delta t \mu \|\nabla \cdot \bE_h^{(n-1)}\|_G^2
\nonumber    \\
&\quad {}+\Delta t\left(\frac{\gamma}{8}\|\be_h^{(n-1)} \|_0^2+\frac{\mu}{4}
  \|\nabla \cdot \be_h^{(n-1)}\|_0^2\right)
 +\frac{\Delta t }{2}C_{\overline\mu}'.
\end{align}
We chose
\begin{equation}
\label{cond_delta}
\delta < \frac{1}{12},
\end{equation}
and
\begin{equation}
\label{cond_Dt}
\Delta t \le \frac{12}{\gamma},
\end{equation}
so that the following inequalities hold:
\begin{align}
\label{ineq1}
\left(1+\frac{\gamma}{12}\Delta t\right) &> \left(1+2\delta^2\gamma\Delta t\right),
\\
\nonumber
\frac{1}{6} &> \frac{1+\frac{\gamma}{12}\Delta t}
{1+2\delta^2\gamma\Delta t }\delta ,
\\
\label{ineq3}
2&> \frac{1+\frac{\gamma}{12}\Delta t}
{1+2\delta^2\gamma\Delta t }.
\end{align}
Thus,
for
\begin{align*}
a_n=&\|\bE_h^{(n)}\|_G^2
     +\frac{1}{1+2\delta^2 \gamma\Delta t}\biggl(\delta \Delta t {\mu}\|\nabla \cdot \bE_h^{(n)}\|_G^2
\nonumber \\
&{}+\Delta t\left(\frac{\gamma}{8}\|\be_h^{(n)} \|_0^2+\frac{\mu}{4}
  \|\nabla \cdot \be_h^{(n)}\|_0^2\right)\biggr),
\end{align*}
from~(\ref{eq:semimp3}) it follows that
\begin{equation}
\label{eq:semimp4}
\left(1+\frac{\gamma}{12}\Delta t\right) a_n \le
\left(1+2\delta^2\gamma\Delta t\right) a_{n-1} + \frac{\Delta t}{2} C_{\overline\mu}'.
\end{equation}

When $\mu=0$, for the third term on the left-hand side of~(\ref{eq:semimp2}),
we write
\begin{align*}
\frac{\nu}{4}\|\nabla  \be_h^{(n)}\|_0^2&=\frac{3}{16}\nu \|\nabla\be_h^{(n)}\|_0^2+
\frac{\nu}{16}\left(\|\nabla \be_h^{(n)}\|_0^2 + \|\nabla \be_h^{(n-1)}\|_0^2
\right) - \frac{\nu}{16}  \|\nabla \be_h^{(n-1)}\|_0^2
\nonumber\\
&{}\ge \frac{3}{16}\nu\|\nabla\be_h^{(n)}\| +
\frac{\nu}{24} \|\nabla\bE_h^{(n)}\|_G^2 -
\frac{\nu}{16}  \|\nabla  \be_h^{(n-1)}\|_0^2
\end{align*}
where, in the last inequality we have argued as in~\eqref{eq:equiv2}.  Now,
besides~\eqref{cond_Dt}, we assume
\begin{equation}
\label{cond_delta_mu0}
\delta \le \frac{1}{48}
\end{equation}
so that, besides~(\ref{ineq1}) and~(\ref{ineq3}),
the following inequality holds
\begin{equation}
\frac{1}{24} > \frac{1+\frac{\gamma}{12}\Delta t}
{1+2\delta^2\gamma\Delta t } \delta,
\end{equation}
Thus, arguing as before, we have that for
\begin{align*}
a_n=&\|\bE_h^{(n)}\|_G^2
     +\frac{1}{1+2\delta^2 \gamma\Delta t}\biggl(\delta \Delta t \nu \|\nabla \bE_h^{(n)}\|_G^2
\nonumber \\
&{}+\Delta t\left(\frac{\gamma}{8}\|\be_h^{(n)} \|_0^2
+\frac{\nu}{16}
  \|\nabla  \be_h^{(n)}\|_0^2\right)\biggr),
\end{align*}
relation~(\ref{eq:semimp4}) holds.

The estimation of~$C_{\overline \mu}'$ is done in Lemma~\ref{le:C0C1_estb}, below,
so that we can conclude with the following result.

\begin{Theorem}\label{th:main_bdf2_semimp}
Under the hypotheses of~Theorem~\ref{Th:main_muno0} but with $\beta\ge 8L'/\delta$ instead of $\beta\ge 8L$ and $H$ satisfying condition
 \eqref{eq:as2b} instead of ~\eqref{eq:as2},  assume also that
 $\Delta t$ satisfies
 \eqref{cond_Dt}.
Fix $\delta \in(0,1/12)$ if
$\mu>0$ or $\delta\in~(0,1/48)$ if $\mu=0$. Then, for the finite element approximation, solution
 of~\eqref{eq:method2_fully_im} when $d_t$ is given by~(\ref{eq:BDF2_1}--\ref{eq:BDF2_2}),
 the following bounds hold:
 \begin{align*}
\| \bu(t_n)-\bu_h^{(n)}\|_0 \le& C\left(\frac{1+2\delta^2\gamma\Delta t}{1+\frac{\gamma }{12}\Delta t} \right)^{(n-1)/2} \|\bu_h(0)-\bu(0)\|_0\nonumber\\
&{}
+\frac{C}{(\nu L)^{1/2}}\left(M_{-1}(\bu_{ttt})+K_1(\bu,|\Omega|)M_0(\bu_{tt})\right)(\Delta t)^2 \\
&{}+
\frac{C}{L^{1/2}}\biggl(\hat C_0 +\frac{1}{\nu^{1/2}}|\Omega|^{(1+\hat r-r)/d}\overline N_{\hat r}(\bu_t,p_t)\biggr)h^{r},
\qquad \mu=0.
\\
\| \bu(t_n)-\bu_h^{(n)}\|_0 \le&C\left(\frac{1+2\delta^2\gamma\Delta t}{1+\frac{\gamma }{12}\Delta t} \right)^{(n-1)/2}\|\bu_h(0)-\bu(0)\|_0
\nonumber\\
&{}
+\frac{C}{L} \left(M_{0}(\bu_{ttt})+K_1(\bu,|\Omega|)M_1(\bu_{tt})\right)(\Delta t)^2\\
&{}+
\frac{C}{L^{1/2}}\biggl(\hat C_1 + \frac{1}{L^{1/2}} M_{r-1}(\bu_t)\biggr)h^{r-1},
\qquad \mu>0.
\end{align*}
where $\gamma$ is defined in \eqref{eq:gamma},
$L$ in~(\ref{eq:L}--\ref{eq:L_muno0}), $\hat C_0$ and~$\hat C_1$
in~\eqref{hatC0} and \eqref{hatC1}, respectively, and $\hat r=r-1$ if $r\ge 3$ and $\Omega$ is of class~${\cal C}^3$ and $\hat r=r$ otherwise.
\end{Theorem}


\begin{remark}\label{re:beta_cond}{\rm  In Theorems~\ref{Th:main_muno0} and~\ref{th:main_bdf2} we assume $\beta\ge 8L$.
Let us observe that in the case $\mu=0$, in view of estimates \eqref{cota_sh_inf} and~\eqref{cota_sh_inf_1}, we have that $\beta\ge 8L$ when $\bw_h=\bs_h$ if
\begin{eqnarray}\label{eq:as1_u}
\beta \ge 8\left(2D_0\left((M_{d-1}(\bu)M_3(\bu))^{1/2} +\bigl(\overline{N}_2(\bu,p)\overline{N }_{d}(\bu,p)\bigr)^{1/2}\right)\nonumber\right.\\
\left.+D_0^2\frac{
M_{d-2}(\bu)M_2(\bu)+\overline{N}_1(\bu,p)\overline{N}_{d-1}(\bu,p)}{4\nu}\right),
\end{eqnarray}
with $D_0$ the constant in~(\ref{cota_sh_inf}--\ref{la_cota}).
In case $\mu\neq 0$ from \eqref{cotainfty1} and \eqref{cota_sh_inf_mu} we have that $\beta\ge 8L$  when $\bw_h=\bs_h^m$ if
\begin{equation}\label{eq:as1_muno0_u}
\beta\ge 16\left(D_1\sup_{\tau\ge 0}\|\nabla \bu(\tau)\|_\infty+D_1^2\frac{
M_{d-2}(\bu)M_2(\bu)}{4\mu} 
\right).
\end{equation}
In Theorem~\ref{th:main_bdf2_semimp} we assume $\beta\ge 8L'/\delta$ which leads to  assumptions on $\beta$ analogous to those above with obvious changes.}
\end{remark}
\begin{lema}\label{le:est-1b} Let $\bu^{(n)}$ and $\tilde \bu^{(n)}$ denote $\bu(t_n)$
and~$\tilde \bu(t_n)$, where $\tilde \bu$ is defined in \eqref{eq:lautilde}, and let~$\bw_h^{(n)}=\bs_h(t_n)$ when $\mu=0$
and~~$\bw_h^{(n)}=\bs_h^m(t_n)$. Then, the following bounds hold
\begin{align}
\label{eq:K_0b}
\sup_{\|\bvar\|_1=1}|& b_h(\tilde \bu^{(n)},\bu^{(n)},\bvar)-b_h(\tilde \bw_h^{(n)},\bw_h^{(n)},\bvar)|
\nonumber \\
&\le K_0(\bu,p,|\Omega|) \left(\|\tilde \bu^{(n)}-\tilde \bw_h^{(n)}\|_0+\|\bu^{(n)}-\bw_h^{(n)}\|_0\right),\quad
\mu=0,
\\
\sup_{\|\bvar\|_0=0}&| b_h(\tilde \bu^{(n)},\bu^{(n)},\bvar)-b_h(\tilde \bw_h^{(n)},\bw_h^{(n)},\bvar)| \nonumber\\
&\le K_1(\bu,|\Omega|) \left(\|\tilde \bu^{(n)}-\tilde\bw_h^{(n)}\|_1+\|\bu^{(n)}-\bw_h^{(n)}\|_1\right), \quad \mu>0,
\label{eq:K_1b}
\end{align}
 $K_0(\bu,p,|\Omega|)$ and~$K_1(\bu,|\Omega|)$ are the quantities
defined in~(\ref{Cup}) and~(\ref{Cu}), respectively.
\end{lema}
\begin{proof}
For simplicity, we denote $\beps=\bu^{(n)}  - \bw_h^{(n)}$ and $\tilde\beps =
\tilde \bu^{(n)}  - \tilde \bw_h^{(n)}$, and drop the explicit dependence on~$n$.
Adding $\pm b_h(\bw_h,\bu,\bvar)$ we have
\begin{equation}
\label{coneps1}
| b_h(\tilde \bu,\bu,\bvar)-b_h(\tilde \bw_h,\bw_h,\bvar)| \le
|b_h(\tilde \bw_h,\beps,\bvar)| + | b_h(\tilde \beps,\bu,\bvar)|.
\end{equation}
For the first term on the right-hand side above, we use the skew-symmetry
property of~$b_h$ to interchange the roles of~$\beps$ and~$\bvar$, so that arguing
as in~\cite[Lemma~5]{proyNS} we have
\begin{equation}
\label{contilde}
|b_h(\tilde \bw_h,\bvar,\beps)| \le  
\|\tilde \bw_h\|_\infty\|\beps\|_0\|\nabla \bvar\|_0+
C\|\nabla\tilde \bw_h\|_{L^{2d/(d-1)}}\|\beps\|_0\|\bvar\|_{L^{2d}},
\end{equation}
and, for the second term on the right-hand side of~(\ref{coneps1}) we write
$$
  b_h(\tilde \beps,\bu,\bvar) =\frac{1}{2}\left((\tilde\beps\cdot\nabla \bu,\bvar) -(\tilde\beps\cdot\nabla \bvar,\bu)\right),
$$
so that we have
\begin{align*}
 | b_h(\tilde \beps,\bu,\bvar)| \le& \frac{1}{2}\|\tilde\beps \|_0\|\nabla \bu\|_{L^{2d/(d-1)}}
 \|\bvar\|_{L^{2d}} + \frac{1}{2}\|\tilde\beps \|_0 \|\bu\|_{\infty} \|\nabla\bvar\|_0.
\end{align*}
To bound~$\|\nabla\bu\|_{L^{2d/(d-1)}}$ and~$\|\bu\|_\infty$ we apply \eqref{eq:parti_ineq} and~\eqref{eq:agmon}, respectively,
and applying~Sobolev's inequality~\eqref{sob1} we have~$\|\bvar\|_{L^{2d}}\le c_1|\Omega|^{(3-d)/(2d)}\|\bvar\|_1$.
Thus, the proof of~(\ref{eq:K_0b}) will be complete after bounding the factors
$\|\tilde \bw_h\|_\infty$ and $\|\nabla\tilde \bw_h\|_{L^{2d/(d-1)}}$
featuring in~(\ref{contilde}), but an upper bound of those factors follows directly
from~(\ref{cota_sh_inf}) and~(\ref{la_cota}).

To prove~(\ref{eq:K_1b}),  we interchange the roles of~$\varphi$ and~$\beps$ and
$\tilde\beps$
 in the arguments above, and using estimates~\eqref{cota_sh_inf_mu}
and~\eqref{la_cota_mu} to bound~$\|\tilde \bw_h\|_\infty$ and~$\|\nabla\tilde \bw_h\|_{L^{2d/(d-1)}}$, respectively.
\end{proof}

 \begin{lema}\label{le:C0C1_estb} For $C_0'$ and~$C_1'$ defined in~(\ref{C0'}) and~(\ref{C0'}), respectively,
 the following bounds hold
\begin{align}
\label{eq:cotaC0'}
C_0' \le & \frac{C}{\nu}\left(M_{-1}^2(\bu_{ttt})+K_1^2(\bu,|\Omega|)M_0^2(\bu_{tt})\right)(\Delta t)^4
\nonumber\\
&{}+ \hat C_0^2(r,\beta,\nu,|\Omega|,\bu,p)h^{2r}
\\
\label{eq:cotaC_1'}
C_1' \le & \frac{C}{L}\left(M_{0}^2(\bu_{ttt})+K_1^2(\bu,|\Omega|)M_1^2(\bu_{tt})\right)(\Delta t)^4
\nonumber\\
&{}
+\hat C_1^2(r,h,\beta,\mu,|\Omega|,\bu,p) h^{2r-2}
\end{align}
where~$\hat C_0$ and~$\hat C_1$ are the constants defined in~(\ref{hatC0}),
(\ref{hatC1}), respectively, and~$K_1$ is the constant defined in~\eqref{Cu}.
\end{lema}
\begin{proof}
In view of how $\hat L$ is defined in~\eqref{eq:L1b} and how~$L$ and $L'$ are defined in~(\ref{eq:L}--\ref{eq:L_muno0})
and~(\ref{eq:Lb}--\ref{eq:L_muno0b}), respectively,
we have that
$$
 L/2 \le L'\le L,
$$
so that the factor $1/L'$ in~(\ref{C1'}) can be bounded by~$2/L$.

The estimation of~$\bu_t^{(n)} - d_t \bu^{(n)}$ is that of the fully implicit BDF2,
so that it gives rise to the terms involving time derivatives of~$\bu$ in the estimates
in Theorem~\ref{th:main_bdf2}. The rest of the terms in~$C_0'$ and~$C_1'$ have
already been estimated in~Lemma~\ref{le:C0C1_est}, except the second
term on the right-hand side of~\eqref{tau1b}. To estimate this term, with $\tilde \bu$ defined in \eqref{eq:lautilde},
we write
\begin{align}
\label{eq:truncb}
&b_h(\bu^{(n)},\bu^{(n)},\bvar_h)-b_h(\tilde \bw_h^{(n)},\bw_h^{(n)},\bvar_h)\\
&\quad{}=b_h(\bu^{(n)}-\tilde \bu^{(n)},\bu^{(n)},\bvar_h)
+b_h(\tilde \bu^{(n)},\bu^{(n)},\bvar_h)
-b_h(\tilde \bw_h^{(n)},\bw_h^{(n)},\bvar_h).\nonumber
\end{align}
The second term on the right-hand side above is estimated in
Lemma~\ref{le:est-1b}.
For the first one, to obtain the corresponding estimate in $H^{-1}$ we proceed
as follows.
For~$\bvar\in H^{1}_0(\Omega)^d$, using the skew-symmetry property of~$b_h$
we write
\begin{align*}
\left|b_h(\bu^{(n)}-\tilde \bu^{(n)},\bu^{(n)},\bvar)\right| =&
\left|b_h(\bu^{(n)}-\tilde \bu^{(n)},\bvar,\bu^{(n)})\right|
\nonumber\\
{}\le & \|\bu^{(n)}-\tilde \bu^{(n)}\|_0\|\nabla\bvar\|_0\| \|\bu^{(n)}\|_\infty
\nonumber\\
{}\le &  \|\bu^{(n)}-\tilde \bu^{(n)}\|_0\|\nabla\bvar\|_0 c_A(M_{d-2}(\bu)M_2(\bu))^{1/2},
\end{align*}
where in the last inequality we have applied~\eqref{eq:agmon}.
Now Taylor expansion easily reveals that $\|\bu^{(n)}-\tilde \bu^{(n)}\|_0
\le CM_0(\bu_{tt})(\Delta t)^2$ and the proof of~\eqref{eq:cotaC0'} is finished.
To obtain an $L^2$-estimate corresponding to the first term on the right-hand side of~\eqref{eq:truncb}, for
$\bvar\in L^2(\Omega)^d$ we write
\begin{align*}
\left|b_h(\bu^{(n)}-\tilde \bu^{(n)},\bu^{(n)},\bvar)\right|
\le \|\bu^{(n)}-\tilde \bu^{(n)}\|_{L^{2d}}
 \|\nabla\bu\|_{L^{2d/(d-1)}} \|\bvar\|_{0}.
\end{align*}
To bound~$\|\nabla\bu\|_{L^{2d/(d-1)}}$ we apply \eqref{eq:parti_ineq},
and applying~Sobolev's inequality~\eqref{sob1} we have
$$
\|\bu^{(n)}-\tilde \bu^{(n)}\|_{L^{2d}}\le c_1|\Omega|^{(3-d)/(2d)}\|\bu^{(n)}-\tilde \bu^{(n)}\|_1,
$$
and the proof of~\eqref{eq:cotaC_1'} is finished.

\end{proof}

\subsection{The Lagrange interpolant}\label{sub_la}

In this section we consider the case in which $I_H \bu=I_H^{La} \bu$ is the Lagrange interpolant. The proof of the following lemmas can be found in \cite[Lemma 310, Lemma 3.11]{Ours}
\begin{lema}
Let $\bv_h\in X_{h,r}$ then the following bound holds
\begin{equation}\label{eq:cotainter_la2}
\|\bv_h-I_H^{La}\bv_h\|_0\le c_{La} H\|\nabla \bv_h\|_0,
\end{equation}
where
\begin{equation}\label{la_c_La}
c_{La}= C\left({H}/{h}\right)^{\frac{d(p-2)}{2p}},
\end{equation}
where $C$ is a generic constant and $p=3$ if $d=2$ and $p=4$ if $d=3$.
\end{lema}
\begin{lema}\label{le:(I-I_h)(u-s_h)} Let $\bw_h$ be the Stokes projection defined in \eqref{stokesnew} in case $\mu=0$ or the modified Stokes projection defined in \eqref{stokespro_mod_def} in case $\mu>0$. Then, the following bounds hold
\begin{eqnarray*}\label{eq:trun_la_stokes}
\|(I-I_H^{La})(\bw_h-\bu)\|_0&\le& C H^2 h^{r-2}{N}_r(\bu,p),\quad \mu=0\\
\|(I-I_H^{La})(\bw_h-\bu)\|_0&\le& C H^2 h^{r-2}\|\bu\|_r,\quad \mu>0.
\end{eqnarray*}
where $C$ is a generic constant.
\end{lema}
Assuming $H/h$ remains bounded one can apply \eqref{eq:cotainter_la2} instead of \eqref{eq:cotainter}. Arguing as in \cite[Theorem 3.12]{Ours}
and applying Lemma~\ref{le:(I-I_h)(u-s_h)} one can replace the term $\frac{\beta }{2}c_0^2\|\bu^{(n)}-\bw_h^{(n)}\|_0^2$ in Lemmas~\ref{lema_general} and~\ref{lema_generalb} by
$\frac{\beta}{2}(C h^r \overline{N}_r(\bu,p)+\|\bu^{(n)}-\bw_h^{(n)}\|_0)^2$ in case $\mu=0$ and by
$\frac{\beta}{2}(C h^r {M}_r(\bu)+\|\bu^{(n)}-\bw_h^{(n)}\|_0)^2$ in case $\mu>0$. Then, Theorems~\ref{Th:main_muno0}, \ref{th:main_bdf2}
and \ref{th:main_bdf2_semimp}  hold  with obvious changes.

\section{Numerical experiments}\label{se:num}

We present some numerical experiments to check the results of the previous section.
Following standard practice, we use an example with a known solution. In particular, we consider
 the Navier-Stokes equations in $\Omega=[0,1]^2$, with the forcing term~$\bff$ chosen so that the solution $\bu$ and~$p$ are given by
\begin{eqnarray}
\label{eq:exactau}
\bu(x,y,t)&=& \frac{6+4\cos(4t)}{10} \left[\begin{array}{c} 8\sin^2(\pi x) (2y(1-y)(1-2y)\\
-8\pi\sin(2*\pi x) (y(1-y))^2\end{array}\right]\\
p(x,y,t)&=&\frac{6+4\cos(4t)}{10} \sin(\pi x)\cos(\pi y).
\label{eq:exactap}
\end{eqnarray}
In the results below, the spatial discretization was done with $P_2/P_1$ elements
on regular triangulations with SW-NE diagonals. For the coarse mesh interpolation we take the Cl\'ement interpolant on piecewise
constants. The time integration was done with semi-implicit BDF2.
In
what follows the initial condition was set to~$\bu_h={\bf 0}$ and~$p=0$, so that there is
an $O(1)$ error at time $t=0$. The value of the nudging parameter was set to~$\beta=1$.

For $\nu=10^{-2}$, $10^{-4}$, $10^{-6}$ and~$10^{-8}$, we computed approximations on triangulations with mesh size $h=1/12$, $h=1/24$ and~$h=1/48$, and different values of
$\Delta t$, while the coarse mesh size was set to~$H=3h$. In Fig.~\ref{fig0} we show
relative errors in velocity for $\nu=10^{-6}$, $\mu=0.05$, corresponding to
different combinations of $\Delta t$ and~$h$, plotted against~$t$.
Recall the error bound for $\mu>0$ in Theorem~\ref{th:main_bdf2_semimp}, where on
the right-hand side we have the initial error times a term decaying exponentially,
an~$O((\Delta t)^2)$ term and an $O(h^{r-1})$ term ($r=3$ in the present case).
\begin{figure}[h]
\begin{center}
\includegraphics[height=5truecm]{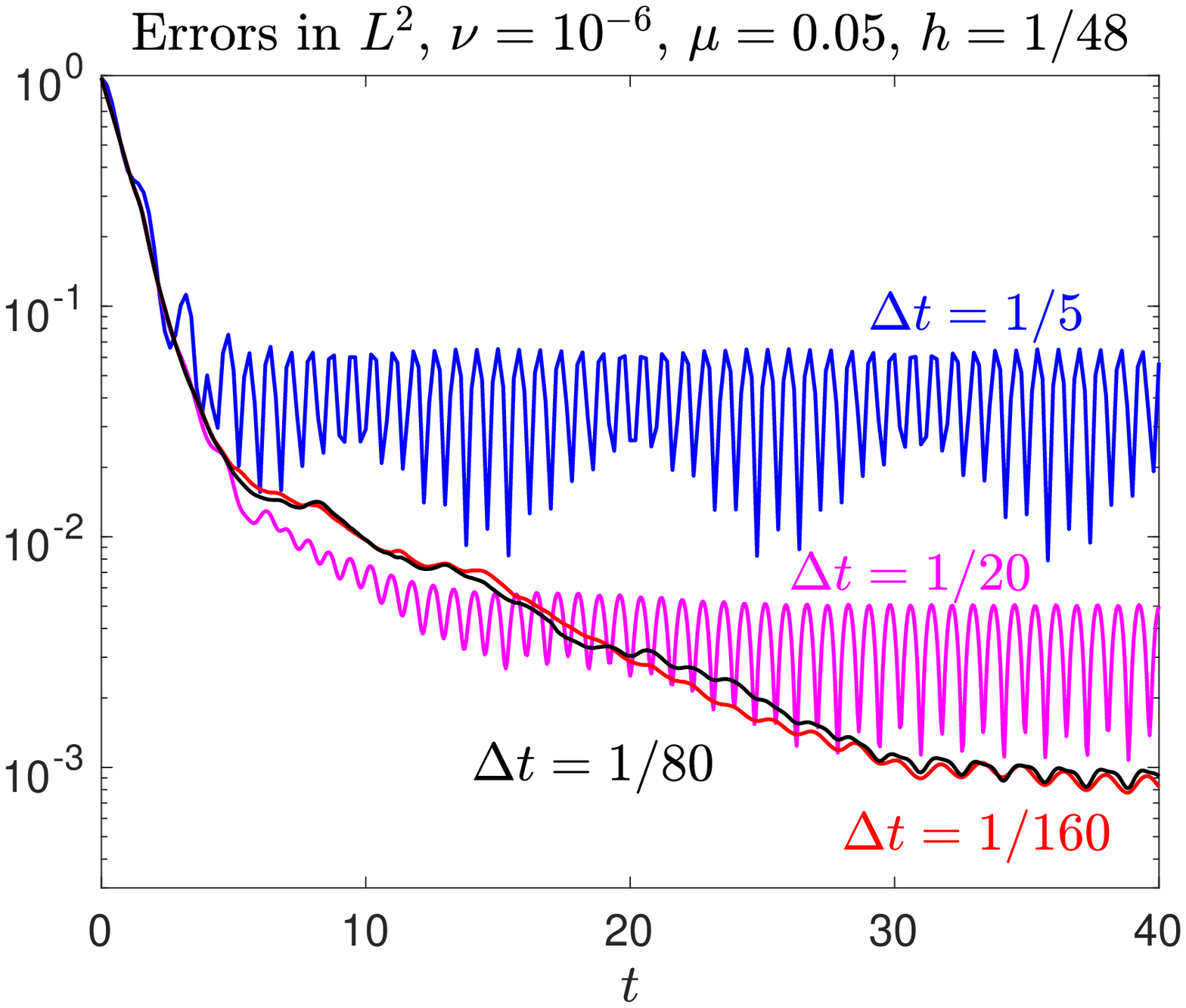}\,
\includegraphics[height=5truecm]{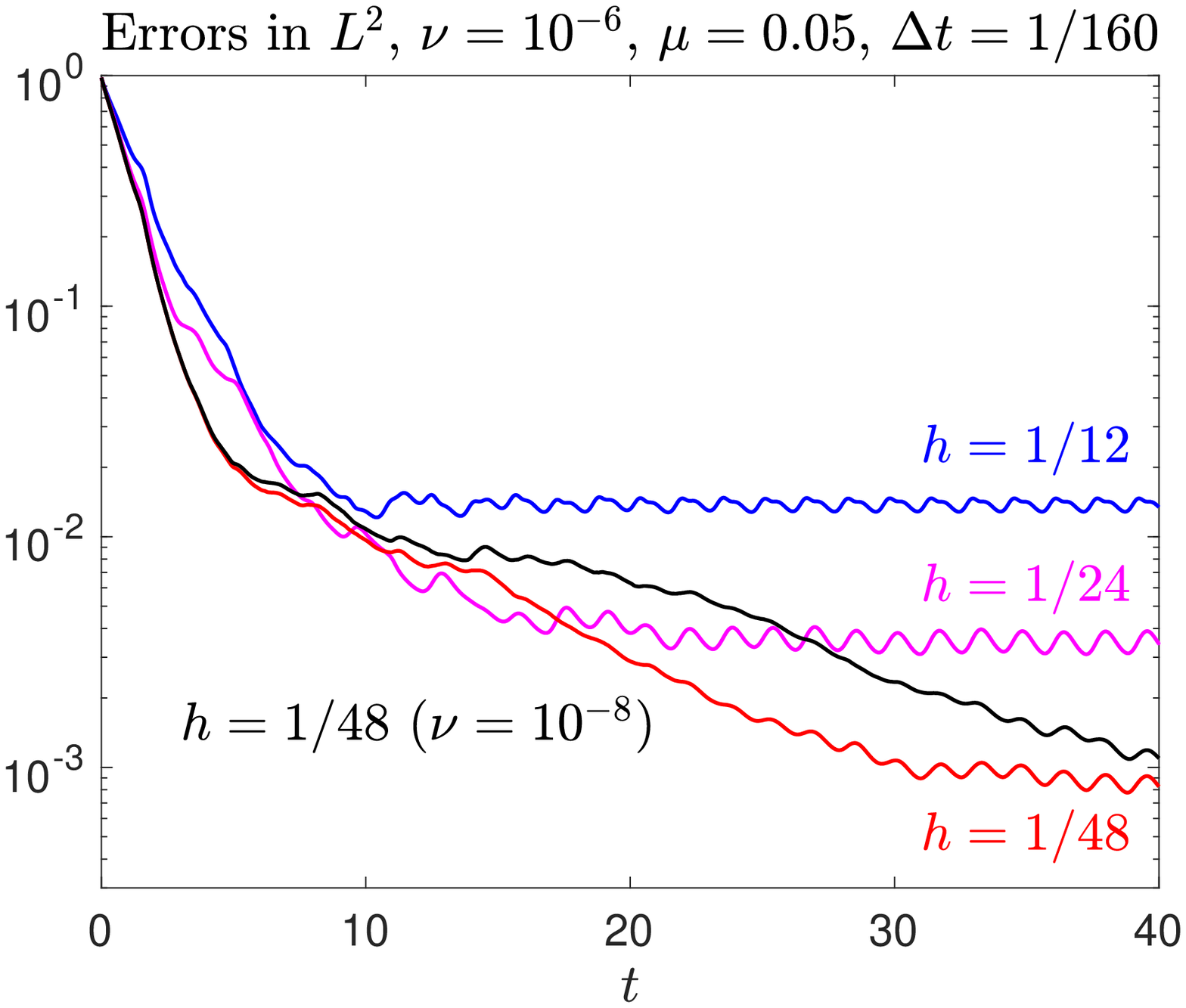}
\end{center}
\caption{Velocity errors vs $t$ for~$\nu=10^{-6}$.\label{fig0}}
\end{figure}
In Fig.~\ref{fig0}, it can be seen how the error at the initial time is equal
to~$1$, decays exponentially with time until reaching the asymptotic regime, where,
in this example, its value oscillates periodically. On
the left plot, we show the errors for $h=1/48$ and decreasing values of~$\Delta t$.
In the asymptotic regime, the $O((\Delta t)^2)$ term dominates the error for the
two largest values of $\Delta t$. For the two smallest values of~$\Delta t$, the errors
are almost identical in the asymptotic regime, meaning that it is the $O(h^2)$ term
that dominates the error.
On the right plot in Fig.~\ref{fig0}, on the contrary, we show the errors for different values of~$h$ but with $\Delta t$ fixed to
$\Delta t =1/160$, so that the $O(h^2)$ term in the error is dominant in the asymptotic regime. Observe that, for~$h=1/48$, the asymptotic regime is not reached until $t=35$
approximately.
 We obtained similar figures (not shown here) for the rest of the values of
$\nu$, and in all of them we observed that the asymptotic regime is already reached
by $t=35$, except for $\nu=10^{-8}$ and~$h=1/48$ (also shown in~Fig~\ref{fig0}), where it was not reached until
$t=42$. For this reason, in the figures that follow, we computed the maximum
value of the $L^2$ errors in velocity for values of $t_n$ in the interval~$[35, 40]$,
except for~$\nu=10^{-8}$ and $h=1/48$, which they were on the interval~$[42,45]$.

 In Fig.~\ref{fig1} we
present velocity errors in~$L^2$ for the four values of~$\nu$.
For every mesh, the errors obtained with the different values of~$\Delta t$
are plotted with crosses for the results corresponding to~$\mu=0.05$ and
with circular bullets for those corresponding to~$\mu=0$, and, for each mesh,
the results of the different values of~$\Delta t$ are
joined by straight segments, with continuous lines for $\mu=0.05$
and discontinuous lines for~$\mu=0$.
\begin{figure}[h]
\begin{center}
\includegraphics[height=5truecm]{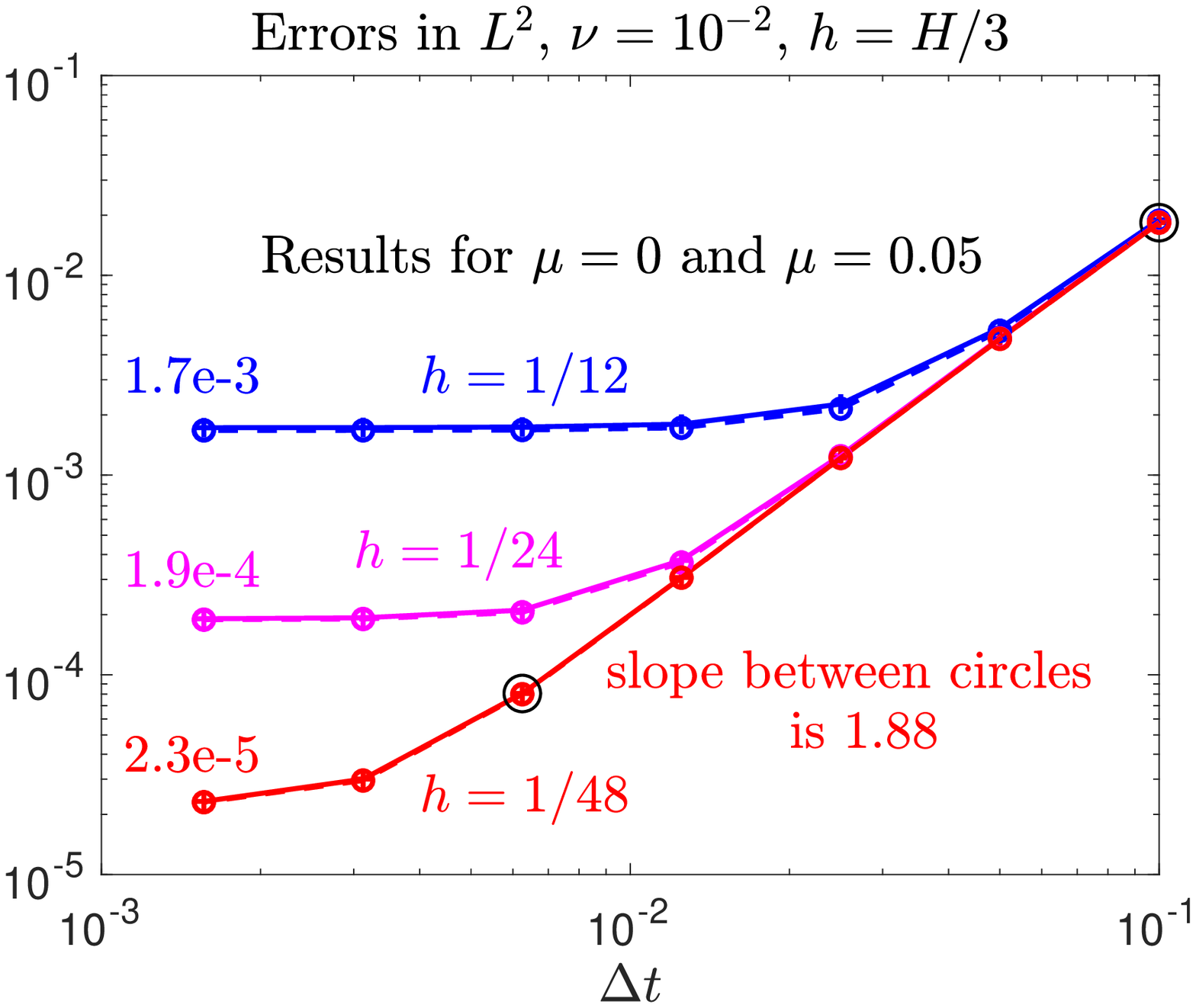}\,
\includegraphics[height=5truecm]{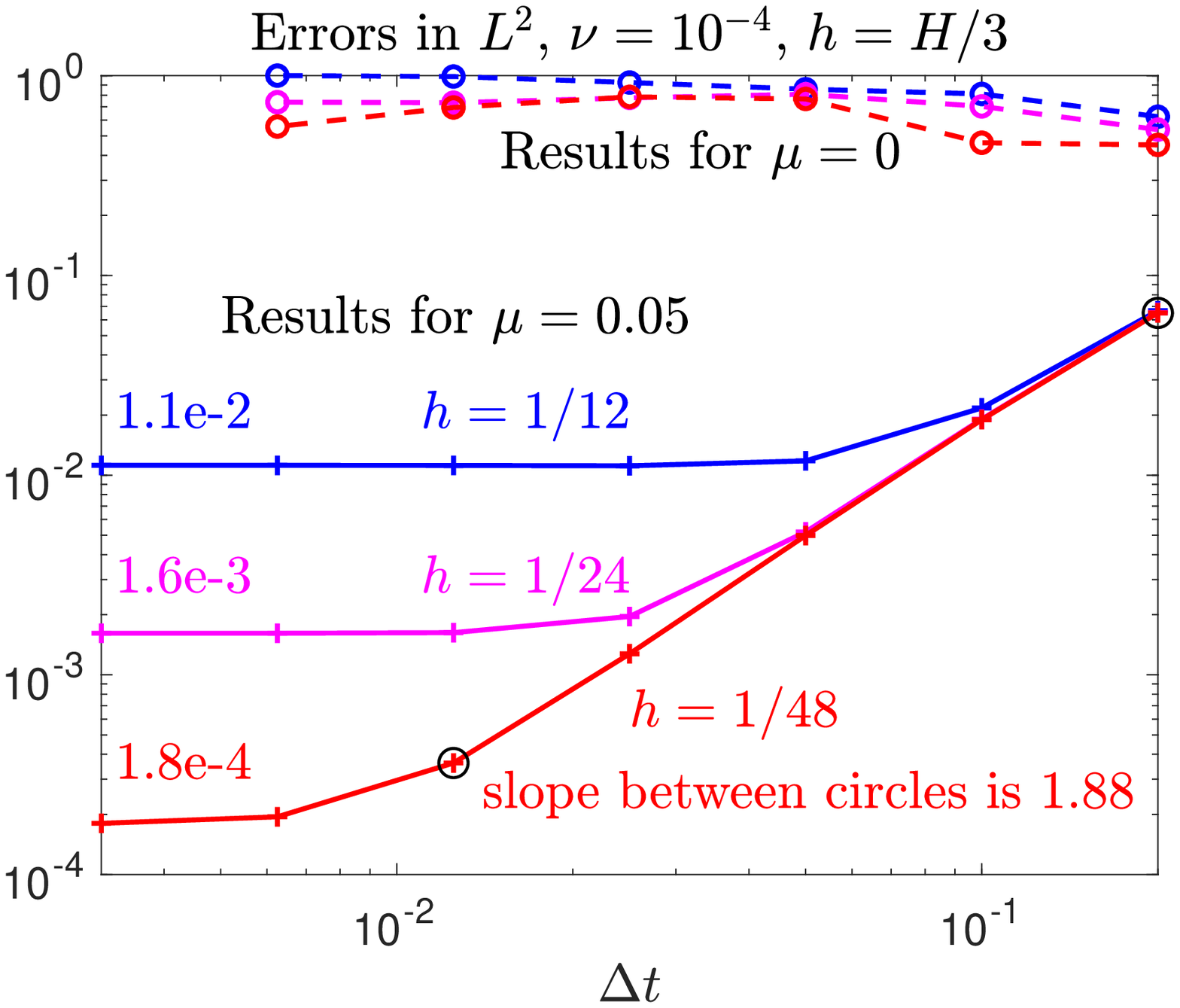}
\medskip
\includegraphics[height=5truecm]{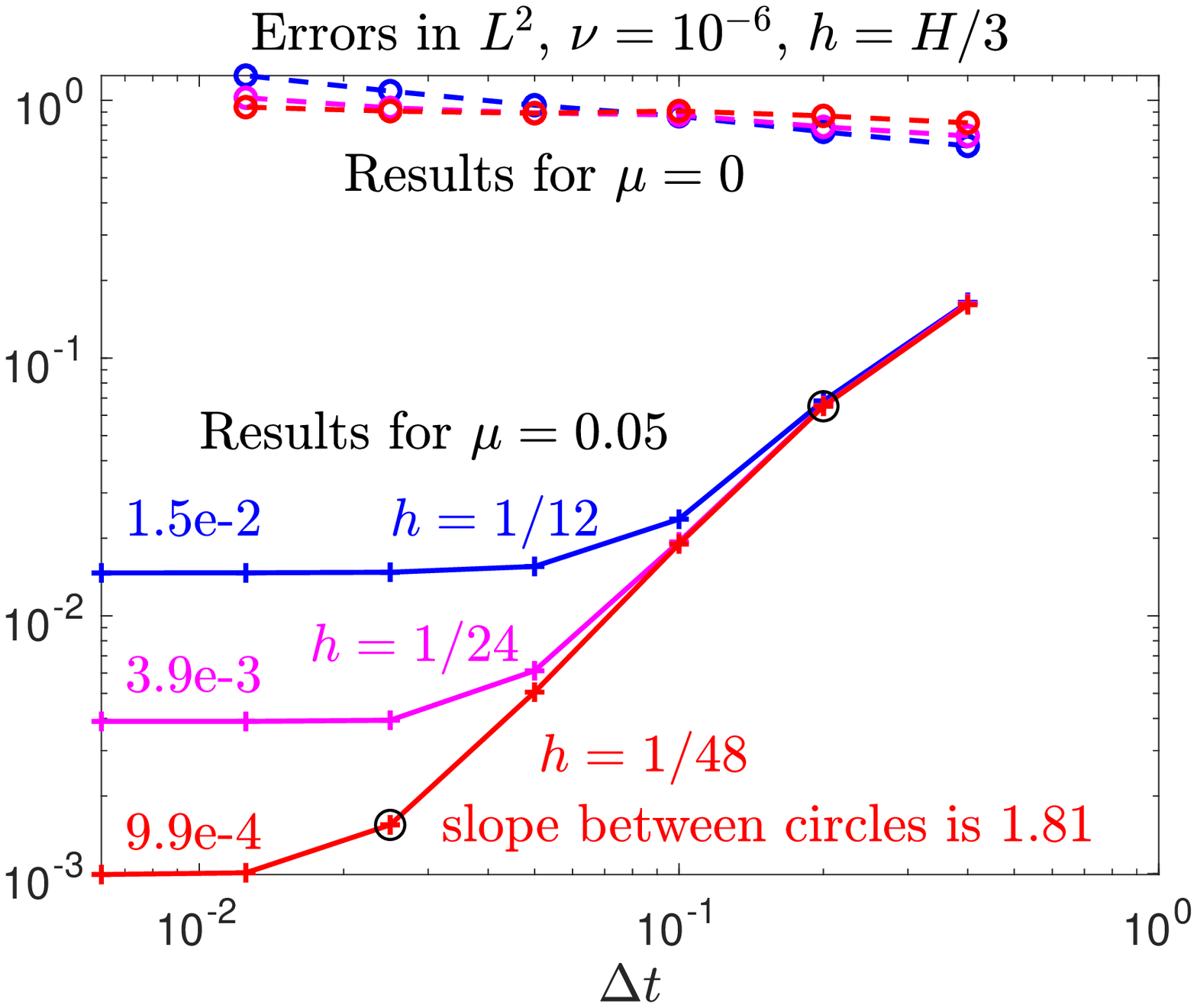}\,
\includegraphics[height=5truecm]{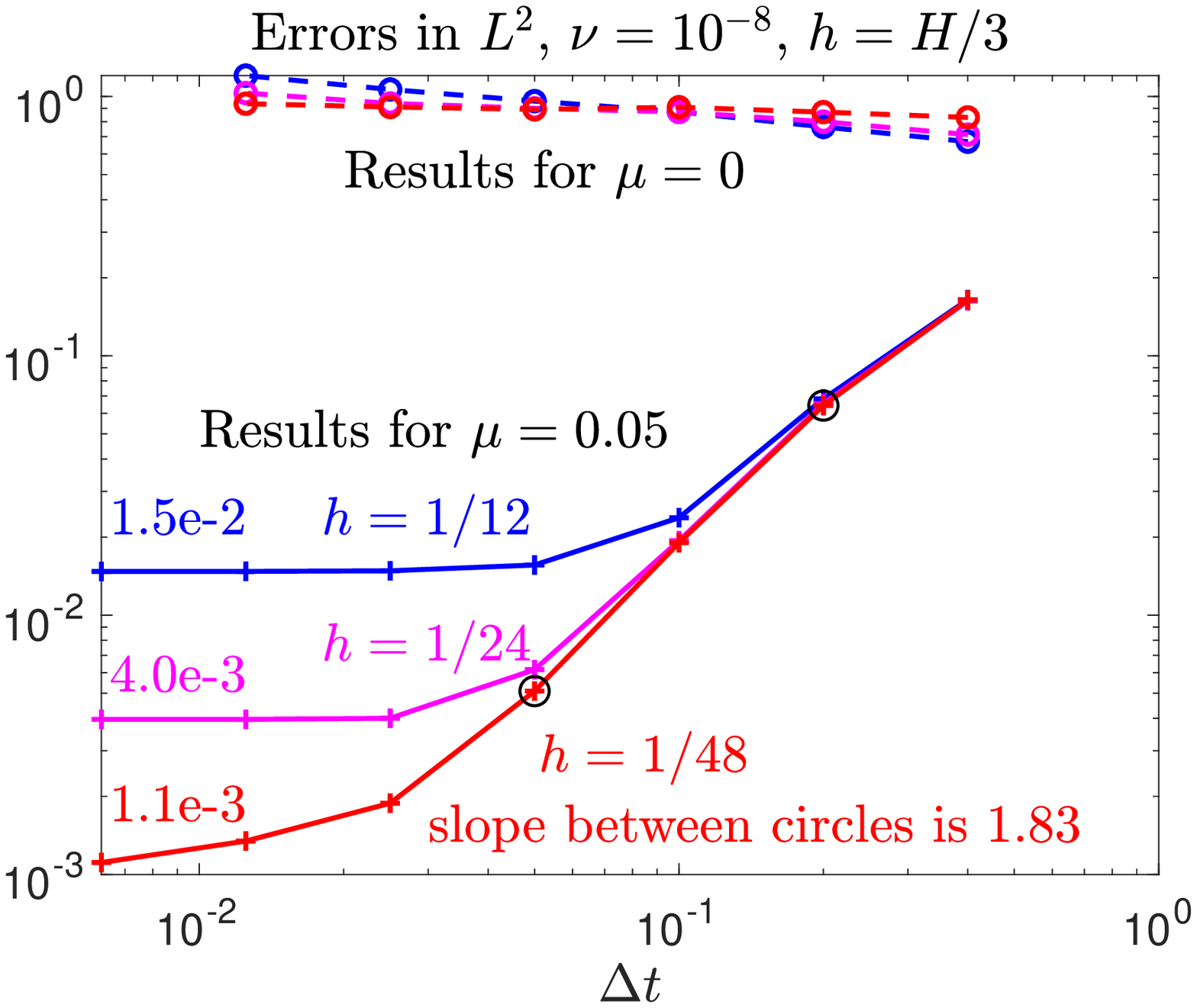}
\end{center}
\caption{Velocity errors vs $\Delta t$.\label{fig1}
}
\end{figure}
To check the order of convergence in time we show the slope of a
least squares fit to the results corresponding to~$h=1/48$, for the
values of~$\Delta t$ between the points marked with black circles. It can
be seen that all slopes have values between 1.81 and~$1.88$, confirming
the $O((\Delta t)^2)$ behaviour of the error whenever the error arising from
time integration dominates that arising from spatial discretization.
To check the order of convergence in space, we show the error
corresponding to~$\mu=0.05$ obtained on every mesh with the smallest
value of~$\Delta t$ used,
which, as commented above, we make sure it was sufficiently small so that the spatial
error dominates.  It can be seen that the error for~$\mu=0.05$ behaves as $O(h^3)$
for $\nu=10^{-2}$ and~$10^{-4}$, and $O(h^2)$ for $\nu=10^{-6}$ or smaller, confirming the second statement in~Theorem~\ref{th:main_bdf2_semimp}
and Remark~\ref{re:nu_large}. Furthermore, comparing the results for
$\nu=10^{-6}$ and~$\nu=10^{-8}$, we see that the (spatial) errors are practically
the same, confirming that the error constants in the
second statement in~Theorem~\ref{th:main_bdf2_semimp}
are independent of~$\nu^{-1}$. The only difference
that we have found,
as shown in Fig.~\ref{fig0}, is the slower rate of decay in time of the error in the initial condition for~$\nu=10^{-8}$.

With respect to the results corresponding to~$\mu=0$, we see a very different
behaviour depending on the size of~$\nu$. For~$\nu=10^{-2}$, they are practically
the same as those  corresponding to~$\mu=0.05$ and, hence, they show an
$O(h^3+(\Delta t)^2)$ behaviour, confirming the first statement
in~Theorem~\ref{th:main_bdf2_semimp}. For smaller values of~$\nu$, however,
the negative powers of~$\nu$ in the error bounds prevent the method from exhibiting
convergence for the values of~$h$ and~$\Delta t$ shown
in~Fig.~\ref{fig1} (presumably, convergence will be achieved for much smaller values
of~$h$). Fig.~\ref{fig1} clearly shows the beneficial effect of the grad-div term when
$\nu$ is small.

\section{Conclusions} We have obtained error bounds for fully discrete approximations with inf-sup stable mixed finite element methods in space
of a continuous downscaling data assimilation method for the two and three-dimensional Navier-Stokes equations. In the data assimilation algorithm measurements on a coarse mesh are given represented by different types of interpolation operators $I_H \bu$, where $I_H$ can be an interpolant for non smooth functions or a standard Lagrange interpolant. To our knowledge, only reference \cite{Ours} and the present paper consider the last case, since in previous references explicit use is made of bounds \eqref{eq:L^2inter} and~\eqref{eq:cotainter}, which are not valid for nodal (Lagrange) interpolation. In the method, a penalty term is added with the aim of driving the approximation towards the solution $\bu$ for which the measurements are known. For the time discretization we consider three different methods: the implicit Euler method and an implicit and a semi-implicit second order backward differentiation formula.For the spatial discretization we consider both the Galerkin method and the Galerkin method with grad-div stabilization.

Uniform error bounds in time have been obtained for the approximation to the velocity field $\bu$ for all the methods, extending the results in \cite{Ours} where the semi-discretization in space is considered.  For the Galerkin method the spatial bounds we prove are optimal, the rate of convergence of the method in $L^2$ being $r$ when using piecewise polynomials of degree $r-1$ in the velocity approximation. In the case where grad-div stabilization is added, the constants in the error bounds do not depend on inverse powers of the viscosity parameter $\nu$, which  is of importance in many applications where viscosity is orders of magnitude smaller than the velocity. For the Galerkin method with grad-div stabilization a rate of convergence $r-1$ is obtained in the $L^2$ norm of the velocity. This bound is sharp, as it is shown in the numerical experiments of the paper. Moreover, it can be clearly observed in the experiments, that for values of the viscosity smaller than $\nu=10^{-4}$ the Galerkin method does not achieve convergence in the range of values of the mesh size for which the Galerkin method with grad-div stabilization converges clearly with the predicted rate of convergence. It is thus to be remarked the dramatic effect of adding grad-div stabilization when the viscosity is small.

In the present paper, as in \cite{Ours}, as opposed to previous references, we do not demand any upper bound on the nudging parameter $\beta$. The authors of  \cite{Larios_et_al} had
 observed (see \cite[Remark 3.8]{Larios_et_al}) that the upper bound on~$\beta$ they required in the analysis does not hold in the numerical experiments, which is also corroborated by the numerical experiments in \cite{Ours}.

\bibliographystyle{abbrv}

\bibliography{references}

\end{document}